\documentclass[12pt]{amsart}
\newif\ifdraft\draftfalse

\usepackage{mathtools, microtype, mathrsfs, xfrac, amssymb, enumerate, xspace}
\usepackage{amsmath, amsthm, url}
\usepackage[latin1]{inputenc}
\usepackage{tikz-cd}
\tikzset{
  symbol/.style={
    draw=none,
    every to/.append style={
      edge node={node [sloped, allow upside down, auto=false]{$#1$}}}
  }
}

\ifdraft\usepackage{showkeys}\else\fi

\usepackage[colorlinks=true,linkcolor=blue,urlcolor=blue,citecolor=blue,pagebackref]{hyperref}
\usepackage{xcolor}

\allowdisplaybreaks[1]

\newtheorem*{namedtheorem}{\theoremname}
\newcommand{\theoremname}{testing}

\newtheorem{theorem}[equation]{Theorem}
\newtheorem{proposition}[equation]{Proposition}
\newtheorem{proposition-definition}[equation]{Proposition-Definition}
\newtheorem{corollary}[equation]{Corollary}
\newtheorem{lemma}[equation]{Lemma}

\newtheorem{nec-cond}[equation]{Necessary Condition}

\theoremstyle{definition}
\newtheorem{definition}[equation]{Definition}

\newtheorem{conjecture}[equation]{Conjecture}
\newtheorem{example}[equation]{Example}

\newtheorem{remark}[equation]{Remark}

\theoremstyle{remark}

\newcommand\arr{\ifinner\to\else\longrightarrow\fi}

\newcommand\arrto{\ifinner\mapsto\else\longmapsto\fi}

\newcommand\Dcech{{\widecheck D}}

\newcommand\im{\operatorname{im}}

\def\displaytimes_#1{\mathrel{\mathop{\times}\limits_{#1}}}

\def\displayotimes_#1{\mathrel{\mathop{\bigotimes}\limits_{#1}}}

\newcommand\aut{\operatorname{Aut}}

\newcommand\spec{\operatorname{Spec}}

\newcommand\Spec{\operatorname{Spec}}

\newcommand{\fr}{\operatorname{Frac}}

\newlength{\ignora}

\renewcommand{\setminus}{\smallsetminus}

\newcommand\af{\mathbb{A}^f}
\newcommand\afp{\mathbb{A}^{f,p}}

\newcommand\Sp{\mathbf{Sp}}

\DeclareFontFamily{U}{mathx}{\hyphenchar\font45}
\DeclareFontShape{U}{mathx}{m}{n}{
      <5> <6> <7> <8> <9> <10>
      <10.95> <12> <14.4> <17.28> <20.74> <24.88>
      mathx10
      }{}
\DeclareSymbolFont{mathx}{U}{mathx}{m}{n}
\DeclareFontSubstitution{U}{mathx}{m}{n}
\DeclareMathAccent{\widecheck}{0}{mathx}{"71}
\DeclareMathAccent{\wideparen}{0}{mathx}{"75}

\renewcommand{\epsilon}{\varepsilon}

\newcommand{\sslash}{\mathbin{/\mkern-6mu/}}
\newcommand\ad{\mathrm{ad}}
\newcommand\sep{\mathrm{sep}}

\DeclareMathOperator\coh{\mathrm{H}}
\DeclareMathOperator\Gal{\mathrm{Gal}}
\DeclareMathOperator\Par{\mathrm{Par}}
\DeclareMathOperator\Dyn{\mathrm{Dyn}}

\DeclareMathOperator\Sym{\mathrm{Sym}}
\DeclareMathOperator\ed{\mathrm{ed}}

\DeclareMathOperator\car{\mathrm{char}}

\DeclareMathOperator\Res{\mathrm{Res}}

\DeclareMathOperator\Aut{\mathrm{Aut}}

\newcommand{\cha}{\operatorname{char}}

\DeclareMathOperator\Mod{\mathrm{Mod}}

\renewcommand{\geq}{\geqslant}
\renewcommand{\leq}{\leqslant}

\newcommand{\bs}{\backslash}
\newcommand{\ovl}{\overline}
\newcommand{\mc}{\mathcal}
\newcommand{\mb}{\mathbb}
\newcommand{\mf}{\mathbf}

\begin{document}

\title[Fixed points, local monodromy and incompressibility]{Fixed
  points, local monodromy, and incompressibility of congruence covers}

%% Author Info 
\author{Patrick Brosnan}
\address[Brosnan]{Department of Mathematics\\
1301 Mathematics Building\\
University of Maryland\\
College Park, MD 20742-4015\\
USA}
\email{pbrosnan@umd.edu}
\author{Najmuddin Fakhruddin} \address[Fakhruddin]{School of
  Mathematics, Tata Institute 
  of Fundamental Research, Homi Bhabha Road, Mumbai 400005, INDIA}
\email{naf@math.tifr.res.in}

%%%%%%%%%%%%%%%%%%%%%%%%%%%%%%%%%%%%%%%%%%%%%%%%%%%% 

\begin{abstract}
  We prove a fixed point theorem for the action of certain local
  monodromy groups on \'etale covers and use it to deduce lower bounds
  in essential dimension.  In particular, we give more geometric
  proofs of many (but not all) of the results of the preprint of Farb,
  Kisin and Wolfson, which uses arithmetic methods to prove
  incompressibility results for Shimura varieties and moduli spaces of
  curves. Our method allows us to prove results for exceptional groups,
  and also for the reduction modulo good primes of Shimura varieties
  and moduli spaces of curves.
\end{abstract}

\maketitle
\tableofcontents

\section{Introduction}
\label{sect.intro}
The purpose of this paper is to prove incompressibility results and
lower bounds on essential dimension using a new method for producing
fixed points, which we formulate in Proposition~\ref{p.criterion} and
Theorem~\ref{t.compress} below.  This recovers many, but not all, of
the incompressibility results obtained using arithmetic methods in the
preprint~\cite{FKW19} of Farb, Kisin and Wolfson.  For example, we get
geometric proofs of the incompressibility of congruence covers related to the
moduli space $\mathcal{A}_g$ of principally polarized abelian
varieties (Corollary~\ref{cor.main}) and the moduli space
$\mathcal{M}_g$ of smooth genus $g$ curves (Theorem~\ref{t.incompmg}).
Our method also allows us to prove new results including an extension
of the theorems of~\cite{FKW19} on incompressibility of congruence
covers of locally symmetric varieties associated to certain groups of
type $E_7$. In contrast to the method of~\cite{FKW19}, our approach
also allows us to prove incompressibility of congruence covers of
both $\mathcal{A}_g$ and $\mathcal{M}_g$ over fields of positive
characteristic.

\subsection{Essential dimension}\label{ss.ed}
Essential dimension is a numerical measure of the complexity of
algebraic objects which first appeared in a paper by J.~Buhler and
Z.~Reichstein~\cite{bur} from 1997.  Since then, many different, but
equivalent, ways of looking at it have arisen.  To explain our results
in more detail and to fix terminology, we give a quick review of
the concepts of incompressibility and essential dimension here
in the introduction. However, for more
details, we refer the reader to \S\ref{s.method} and to~\cite{bur, FKW19, merkurjev-survey}.

Let $k$ be any field. By a variety over $k$ we shall mean a reduced,
separated scheme of finite type. If $f:X\to Y$ is a generically \'etale
morphism of varieties with $Y$ integral, then the \emph{essential
  dimension} of $f$, $\ed f$, is defined to be the minimum of the
dimensions of irreducible varieties $Y'$ such that the following
conditions hold:
\begin{enumerate}
\item There is a dominant rational map $Y\dashrightarrow Y'$ and a
  finite \'etale morphism of schemes $X'\to Y'$.
\item Over the generic point of $Y$, the morphism $X\to Y$ is the
  pullback of the morphism $X'\to Y'$.
\end{enumerate}
We will often abuse notation slightly and write $\ed(X \to Y)$, or
even $\ed X$, instead of $\ed f$. For any prime number $p$, we
define $\ed(f;p)$ (or $\ed(X;p)$), the \emph{essential dimension at
  $p$} of $f$, to be the minimum of $\ed g$, where $g$ ranges over
all morphisms of the form $g: X \times_Y Z \to Z$, with $Z$ integral
and $Z \to Y$ a generically finite dominant morphism of degree prime
to $p$ such that $X\times_Y Z$ is reduced. (The above definitions can be extended to allow reducible $Y$
by setting $\ed f = \max_i \ed(X \times_{Y_i} Y \to Y_i)$, where the
$Y_i$ run over the irreducible components of $Y$, and similarly for
$\ed(f;p)$).

If $G$ is a finite group, $Y$ is an integral variety and $f:X \to Y$
is a $G$-torsor, then $f$ is finite \'etale and the above definitions
apply to $X$. A $G$-torsor is said to be \emph{incompressible} if
$\ed X = \dim X$ and \emph{$p$-incompressible} if
$\ed(X;p) = \dim X$. The essential dimension $\ed G$ of $G$ is
defined to be the maximum of $\ed X$ over all $G$-torsors $X$ as
above and we define $\ed(G;p)$ analogously.

\subsection{Incompressibility results}\label{ss.ir} 
The recent paper of Farb, Kisin and Wolfson~\cite{FKW19} mentioned above
proves the incompressibility of a large class of congruence covers of
Shimura varieties.  For example, let $\mathcal{A}_{g,N}$
denote the moduli space of principally polarized abelian varieties of
dimension $g$ with (symplectic) level $N$-structure~\cite[Definition
I.4.4, p.~19]{FaltingsChai}.  If $N\geq 3$ and if $p$ is a prime
number not dividing $N$, then the natural morphisms
$\mathcal{A}_{g,pN}\to\mathcal{A}_{g,N}$ is an
$\Sp_{2g}(\mathbb{F}_p)$-torsor.
% So $\mathcal{A}_{g,pN}$ is a generically free
% $\Sp(2g,\mathbb{F}_p)$-variety.
One of the key theorems
of~\cite{FKW19} is the $p$-incompressibility of $\mathcal{A}_{g,pN}$
or, equivalently, of the cover
$\mathcal{A}_{g,pN}\to\mathcal{A}_{g,N}$.

The results of~\cite{FKW19} were proved using arithmetic methods
involving reduction modulo $p$ and a subtle argument involving
Serre--Tate coordinates.  Their results go far beyond the case of
$\mathcal{A}_{g,pN}$, proving incompressibility for congruence covers
of many interesting varieties related to $\mathcal{A}_g$ such as the
moduli space of curves $\mathcal{M}_g$ and a large class of Hodge type
Shimura varieties (and even certain subvarieties thereof).

Following a suggestion of Z.~Reichstein, our initial goal in this
paper was to recover the $p$-incompressibility of $\mathcal{A}_{g,pN}$
using an elementary (and, by now, standard) geometric criterion for
incompressibility, which we refer to as {\em the fixed point
  method}. As we explain below, the fixed point method allows us not
only to recover some of the results of~\cite{FKW19} for
$\mathcal{A}_g$ but also to prove several new results.

We will give a stronger statement (Proposition~\ref{prop.fixed-point})
of the fixed point method in \S\ref{s.method}.  However, for the
purposes of this introduction it boils down to the following fact:
Suppose $G$ is a finite group, $X$ is a $G$-torsor over an irreducible
variety $Y$, and $\ovl{X}$ is a $G$-equivariant compactification of $X$.
If a finite abelian $p$-subgroup $H \leq G$ of rank $\dim X$ has a
smooth fixed point on $\ovl{X}$, then $X$ is $p$-incompressible.

In~\S\ref{s.crit}, we prove a criterion,
Proposition~\ref{p.criterion}, for the existence of fixed points in a
compactification of a variety $X$, which is presented as a finite
\'etale Galois cover $f:X\to Y$ with Galois group $G$.  This
criterion, which depends on a partial compactification $\ovl Y$ of $Y$
and the local monodromy of the cover $f$ on a partial toroidal
resolution $S$ of $\ovl Y$, leads to Theorem~\ref{t.compress}, which
combines Proposition~\ref{p.criterion} with the fixed point method to
give lower bounds on essential dimension based on local monodromy.
Applying Theorem~\ref{t.compress} to various situations, e.g., where
$Y$ is the moduli space $\mathcal{M}_g[N]$ of genus $g$ curves with
(symplectic) level $N$ structure or where $Y$ is $\mathcal{A}_{g,N}$
or, more generally, a locally symmetric variety, allows us to prove
lower bounds on essential dimension and, in many cases, even
incompressibility results.

\subsection{Contents of the paper} We now explain the
incompressibility results in more detail along with the topics of the
sections in the paper.  Section~\ref{s.method} briefly reviews
essential dimension and the notion of versal and $p$-versal
torsors. It also proves the version of the fixed point method,
Proposition~\ref{prop.fixed-point}, mentioned above.  Most of this
section is due to Z.~Reichstein, and we are grateful to him for
allowing us to use it. 

Section~\ref{s.crit} proves our main general result on essential
dimension and local monodromy, Theorem~\ref{t.compress}.  The rest of
the paper consists essentially of examples.  Section~\ref{s.tqft}
proves the incompressibility of two types of covers related to the
moduli space $\mc{M}_g[N]$ of genus $g$ curves with (symplectic) level $N$
structure.  The first result, Theorem~\ref{t.incompmg}, recovers the
$p$-incompressibility of the cover $\mc{M}_g[pN]\to \mc{M}_g[N]$, which was
proved by arithmetic means in~\cite[Theorem 3.3.2]{FKW19}. Our proof
is characteristic free and works over all fields of characteristic not
dividing $pN$.  The second result, Theorem~\ref{t.quantum}, proves the
incompressibility of certain ``quantum covers'' of $\mc{M}_g[N]$ arising
from the TQFTs constructed in \cite{bhmv}.

Section~\ref{s.sv} proves our main general results on the essential
dimension of congruence covers of locally symmetric varieties.  We
begin by recalling the notion of a locally symmetric variety in
\S~\ref{lss} essentially following the terminology of Ash, Mumford,
Rapoport and Tai~\cite{amrt}.  These are quotients
$\Gamma \backslash D$ of Hermitian symmetric domains by arithmetic
subgroups $\Gamma$.  (See Remark~\ref{r.cpt} for an explanation of why
we prefer to use the language of locally symmetric varieties rather
than the closely related language of Shimura varieties in the context
of this paper.)

The main theorems of Section~\ref{s.sv} are Theorem~\ref{t.main-pure},
which proves a general lower bound on the essential dimension of
congruence covers of locally symmetric varieties in terms of boundary
components, and Theorem~\ref{thm.main2}, which deals with the special
case where $D$ is a tube domain with a zero dimensional rational
boundary component.  In this special tube domain case, our results
often imply $p$-incompressibility.  Both theorems are proved by
applying our main theorem on essential dimension and local
monodromy, Theorem~\ref{t.compress}, to the case where $X\to Y$ is a
congruence cover of locally symmetric varieties and $\overline{Y}$ is
the Baily--Borel compactification of $Y$.  In Corollary~\ref{cor.main}
of \S\ref{ss.ag}, we apply Theorem~\ref{thm.main2} to recover the
result from~\cite{FKW19} on the incompressibility of the congruence
cover $\mathcal{A}_{g,pN}\to\mathcal{A}_{g,N}$. We also extend this
result to fields of positive characteristic in \S\ref{pos} using the
integral toroidal compactifications of Faltings and Chai
\cite{FaltingsChai}.

Section~\ref{s.Ex} considers the problem of constructing
incompressible congruence covers of locally symmetric varieties
analogous to the congruence covers of modular curves with full level
structure.  These are the covers that are called \emph{principal
  $p$-coverings} in~\cite{FKW19}, and much of what we do in
Section~\ref{s.Ex} is motivated by similar considerations
in~\cite{FKW19}.  However, owing to our methods, which depend on
boundary components, we have to deal with issues not present
in~\cite{FKW19}.  Our main result is Theorem~\ref{t.nont}.  As an
application, we can produce $p$-incompressible congruence covers of
certain locally symmetric varieties of type $E_7$ and with Galois
group $\mathbb{G}(\mathbb{F}_{p^n})$, where $\mathbb{G}$ is the simply
connected form of $E_7$ over $\mathbb{F}_p$.  We also recover most of
the results of~\cite{FKW19} on existence of $p$-incompressible
principal $p$-covers for classical groups.  The notable exception is
when the group is of type $A_n$ for $n$ even.  See Remark~\ref{list}.

There are two appendices in which we reproduce (with permission)
results shown to us by other people. Appendix~\ref{vhs} gives an
argument by M.~Nori, which proves a weak version of
Conjecture~\ref{c.vhs} below on essential dimension of variations of
Hodge structure.  Appendix~\ref{a.benson} gives a proof, due to Dave
Benson, that the essential dimension at $p$ of
$\Sp_{2g}(\mathbb{F}_p)$ is $p^{g-1}$.  This shows that there is
a large (in fact, exponential) difference between the essential
dimension of congruence covers
$\mathcal{A}_{g,pN}\to\mathcal{A}_{g,N}$, which are very particular
$\Sp_{2g}(\mathbb{F}_p)$-torsors, and the essential dimension of
the versal torsor, which is $p^{g-1}$.  (See also
Corollary~\ref{cor.versal} for a general lower bound on essential
dimension of groups with non-abelian $p$-Sylow subgroups.)  

Jesse Wolfson has informed us that his student, Hannah Knight, has also, independently, 
computed $\ed(\Sp_{2g}(\mb{F}_p);p)$.  Moreover, she has made 
significant progress towards the goal of computing $\ed(G;p)$ for much more general finite
quasisimple groups of Lie type.  In particular, she has computed it for groups 
 $G=\Sp_{2g}(\mathbb{F}_{p^r})$  for $r>1$ as well as for analogous orthogonal groups.  
 
\subsection{Comparison with work of Farb, Kisin and Wolfson}

The methods of this paper and of~\cite{FKW19}, although very
different, have a key point in common, i.e., the use of elementary
abelian $p$-groups contained in the Galois group of the congruence
covers.  We use these subgroups in a very direct way via the fixed
point method, whereas their appearance in~\cite{FKW19} is slightly
indirect, as wild ramification groups of degenerations of congruence
covers.  These different ways of exploiting such subgroups accounts,
to a large extent, for the fact that although the results of both
papers have a large intersection, there are results provable by each
method not accessible to the other.

For example, the methods of~\cite{FKW19} apply to many congruence
covers of Shimura varieties whose connected components are compact.
As these congruence covers are \'etale (at least when the congruence
subgroups involved are neat, see \S\ref{s.AC}), the fixed point method
cannot possibly apply (as there are no fixed points).  On the other
hand, the methods of~\cite{FKW19} rely on embedding Shimura varieties
in $\mathcal{A}_g$ and the fact that $\mc{A}_g$ has a good integral
model over which the congruence covers degenerate (at primes dividing
the level). Therefore, they apply essentially only to Hodge type
Shimura varieties, and, for example, not to those of type $E_7$ where
our methods yield new incompressibility results. Furthermore, they do
not apply to $\mc{A}_g$ and $\mc{M}_g$ over fields of positive
characteristic, since the congruence covers over such a field cannot
have unequal characteristic degenerations.

It would be interesting to prove $p$-incompressibility for locally
symmetric varieties associated to groups of type $E_6$.  Since such
varieties are not of Hodge type, the methods of~\cite{FKW19} do not
apply, but, since they are not quotients of tube domains, our methods
also do not suffice to prove incompressibility (Remark \ref{list}(3)).

\subsection{A conjecture}\label{ss.conj}

Recall from \cite[\S 1]{DeligneShimuraCorv} that Hermitian symmetric
domains are special examples of the period domains studied in Hodge
theory.  Given any integral variation of Hodge structure $\mathbb{H}$
on a smooth variety $B$, we get an associated period map
$\varphi:B^{\mathrm{an}}\to\Gamma\backslash D$ where $D$ is a period
domain, or, more generally, a Mumford--Tate domain, with generic
Mumford--Tate group $\mathbf{G}$, an algebraic group over
$\mathbb{Q}$, and $\Gamma\leq\mathbf{G}(\mathbb{Q})$ is an arithmetic
lattice depending on the monodromy.  See the paper~\cite{BBT} of
Bakker, Brunebarbe and Tsimerman for this notation and for their main
theorem, which states that the image $Y$ of $\varphi$ has the
structure of a quasiprojective complex variety. We call $Y$ the
\emph{image of the period map}.

Combining the results obtained by our methods with those
of~\cite{FKW19} along with some wishful thinking leads us to guess
that the dimension of the image of the period map should bound the
essential dimension of congruence covers from below.  To make this
explicit, suppose $\mathbb{H}$ is a torsion-free variation of Hodge
structure on a smooth complex variety $B$.  Write
$\mathbb{H}_{\mathbb{Z}}$ for the local system on $B$ corresponding to
$\mathbb{H}$.  Then, for each prime $p$, we get a family of local
systems $\mathbb{H}_{\mathbb{Z}}/p^n$ of free $\mathbb{Z}/p^n$-modules
over $B$.  Moreover, the \'etal\'e spaces of the sheaves
$\mathbb{H}_{\mathbb{Z}}/p^n$ are algebraic.  So it makes sense to
formulate:

\begin{conjecture}\label{c.vhs}
  Suppose $\mathbb{H}$ is a torsion-free integral variation of Hodge
  structure on a smooth irreducible complex variety $B$.  Let $Y$ be
  the image of the period map and set $d=\dim Y$.  There exists an
  integer $N$ such that, if $p$ is a prime number and $n$ is a
  nonnegative integer with $p^n\geq N$, then
  $\ed(\mathbb{H}\otimes_{\mathbb{Z}}\mathbb{Z}/p^n\to B; p)\geq d$.
\end{conjecture}

We could modify the conjecture by replacing
$\ed(\mathbb{H}\otimes_{\mathbb{Z}}\mathbb{Z}/p^n\to B; p)$ with
$\ed(\mathbb{H}_{\mathbb{Z}}\to B)$ suitably defined.  In Appendix~\ref{vhs}, we
give a precise formulation and a proof, due to M.~Nori, of this
modified statement. This is our main evidence for the validity of the
above conjecture beyond the case where the period domain is Hermitian
symmetric.

\subsection{General notation and notational conventions} 
We try to maintain the convention of writing algebraic groups
$\mathbf{G}$ in boldface and abstract groups $G$, as well as Lie
groups, in non-bold.  For an algebraic group $\mathbf{G}$ over a
subring of $\mathbb{R}$, we write $\mathbf{G}(\mathbb{R})_+$ for the
connected component of the identity of the Lie group
$\mathbf{G}(\mathbb{R})$.  If $\mathbf{G}$ is defined over
$\mathbb{Q}$, then we write
$\mathbf{G}(\mathbb{Q})_+:= \mathbf{G}(\mathbb{R})_+\cap
\mathbf{G}(\mathbb{Q})$.  In \S\ref{s.sv}, $\mathbf{G}$ is usually an
adjoint group over $\mathbb{Q}$, but, in~\S\ref{s.Ex}, $\mathbf{G}$ is
usually taken to be simply connected with adjoint group
$\mathbf{G}^{\ad}$.

We warn the reader that we have reversed what seems to be the usual convention
of writing stacks such as the moduli stacks of curves or principally polarized
abelian varieties  in calligraphic script and the associated course moduli
spaces in roman font.  Fortunately, there are very few stacks in the paper
and, with regard to the above moduli stacks, we are usually taking a large 
enough level structure so that the stack and the space coincide.  Still, we 
apologize in advance if this causes confusion.  

\subsection{Acknowledgements} As we already mentioned, this work owes
its existence to a suggestion from Zinovy Reichstein. We are very
grateful to him for this suggestion and for many other smaller, but
still very significant, contributions he generously made to this paper.

Brosnan would also like to thank Michael Rapoport for several useful
discussions about Shimura varieties and toroidal compactifications and
Dave Benson for showing us how to compute the essential dimension of
the group $\Sp_{2g}(\mb{F}_p)$.  He would like to
thank Jesse Wolfson for several useful suggestions and encouraging emails as well as 
Burt Totaro for typo corrections. Moreover, he thanks the Isaac
Newton Institute for hosting the workshop where the conversations with
Benson took place in January of 2020 and the Simons Foundation for a
Collaboration Grant, which helped make it possible to travel before
the lockdown of March 2020.  Fakhruddin would also like to thank
Gregor Masbaum for useful correspondence on TQFTs, Arvind Nair for
useful conversations on Hermitian symmetric domains, and Madhav Nori
for useful discussions related to Conjecture \ref{c.vhs}. He was
supported by the DAE, Government of India, under project no.~RTI4001.

\section{Essential dimension, versality and the fixed point method}
\label{s.method}

\subsection{Essential dimension of \texorpdfstring{$G$}{G}-varieties}

Let $G$ be a finite group and $k$ be a base field. In
\S\ref{sect.intro}, we have defined the essential dimension $\ed X$
of a $G$-torsor $X \to Y$ and also its essential dimension at $p$, 
$\ed(X;p)$. We have also defined $\ed G$, the essential dimension of
$G$, and $\ed(G;p)$, the essential dimension at $p$ of $G$. In
general, these numbers depend on the field $k$, so we sometimes write
$\ed_k(G)$ and $\ed_k(G;p)$ to emphasize this.

If $X \to Y$ is an irreducible $G$-torsor, then $X$ is also an
$H$-torsor for any subgroup $H$ of $G$ (with base $H \backslash X$),
so we may also consider the essential dimension of $X$ as an
$H$-torsor. In this case, we shall write $\ed_G(X)$ and $\ed_H(X)$ if
there is any risk of confusion.

The following lemma is the analogue of \cite[Lemma 2.2]{bur} for
$G$-torsors and is very similar to~\cite[Lemma 2.1.4]{FKW19}.  
\begin{lemma} \label{l.eddef}%
  Let $X \to Y$ be a $G$-torsor (with $Y$ integral), $Y'$ an
  integral variety with a dominant rational map $\phi: Y \to Y'$, and
  $X' \to Y'$ a finite \'etale morphism such that over the generic
  point of $Y$, $X'$ is equal to $Y \times_{Y'} X'$. Then there exists
  a finite \'etale cover $Y''$ of $Y$ such that the rational map
  $Y \dasharrow Y'$ lifts to $Y''$ and there is a structure of
  $G$-torsor on $X'' = Y''\times_{Y'}X' \to Y''$ such that the induced
  rational map $X \to X''$ is $G$-equivariant.
\end{lemma}
We omit the proof since it is essentially the same as the proof
of~\cite[Lemma 2.1.4]{FKW19}.
% \begin{proof}
%   By shrinking $Y$ if necessary we may assume that the rational map
%   $\phi$ is a morphism and $X = Y \times_{Y'} X'$. The $G$-torsor
%   $X \to Y$ corresponds to a homomorphism $\pi_1^{\mathrm{et}}(Y, y) \to G$, where
%   $y$ is any geometric point of $Y$. The finite \'etale map $X' \to Y'$
%   corresponds to a homomorphism $\rho: \pi_1^{\mathrm{et}}(Y',\phi(y)) \to S_d$
%   where $S_d$ is the symmetric group on $d = \deg(X'/Y') = |G|$
%   elements.  The identification of $X$ with $Y\times_{Y'} X'$ induces
%   an inclusion of $G$ in $S_d$. It follows from this that we may take
%   $Y''$ to be the \'etale cover of $Y$ corresponding to the subgroup
%   $\rho^{-1}(G)$ of $\pi_1^{\mathrm{et}}(Y',\phi(y))$.
% \end{proof}

\begin{remark}\label{rem.sbgrp}
  Suppose $X$ is an irreducible $G$-torsor and $H\leq G$.  Then
  $\ed_H(X)\leq \ed_G(X)$ and $\ed_H(X;p)\leq \ed_G(X;p)$ for every
  prime $p$. This follows from Lemma \ref{l.eddef} since if we have
  $Y', \phi$ as in the lemma and $X'$ is a $G$-torsor then we may set
  $Z$ to be $X/H$, $Z'$ to be $X'/H$, and then the rational map
  $ Y \dasharrow Y'$ induces a rational map $ Z \to Z'$ such that $X$
  is equal to $Z \times_{Z'}X'$ over the generic point of $Z$.

\end{remark}

\begin{remark} \label{rem.elementary} Let $k$ be a field of
  characteristic $\neq p$ containing a primitive $p$th root of unity,  
and suppose $G$ is a finite $p$-group. 
  By a theorem of Karpenko and Merkurjev, $\ed G = \ed(G; p)$ is the
  smallest dimension of a faithful linear representation of $G$
  defined over $k$; see~\cite[Theorem 4.1]{km2}.

  Of particular interest to us will be the case where
  $G = (\mathbb Z/p)^r$ is an elementary abelian group of rank
  $r$. Here $\ed G = \ed(G; p) = r$.  This special case predates the
  Karpenko--Merkurjev theorem and is considerably easier to prove;
  see,~\cite[Example 2.6]{reichstein-icm} or~\cite[Example
  3.5]{merkurjev-survey}.

  Now suppose $k$ is an arbitrary field of characteristic $\neq p$ and
 let  $k'$ be the field obtained from $k$ by adjoining a primitive $p$th
  root of unity. Since $[k': k]$ is prime to $p$, we
  have~$\ed_k(G; p) = \ed_{k'}(G; p)$, see~\cite[Remark 4.8]{km2}. In
  particular,
  $\ed \, (\mathbb Z/p)^r \geqslant \ed((\mathbb Z/p)^r; p) = r$ over
  any field $k$ of characteristic $\neq p$. \qed
\end{remark}

\subsection{Versality}
In this section it will be convenient for us to also consider
irreducible varieties $X$ with faithful $G$-actions which are not
free. Since $G$ is finite, $X$ always has a dense affine open
$G$-invariant subset $U$ on which $G$ acts freely. Then $U \to U/G$ is
a $G$-torsor and we set $\ed X := \ed U$ and $\ed(X;p) :=
\ed(U;p)$. It is easy to see that this is independent of the choice of
$U$.

We say that an irreducible $G$-variety $V$ is {\em weakly versal}
(respectively, {\em weakly $p$-versal}) if, for every $G$-torsor
$X \to Y$, with $Y$ integral, there exists a $G$-equivariant rational map
$X \dasharrow V$ (respectively a $G$-equivariant correspondence
$X \rightsquigarrow V$ of degree prime to $p$).  Here $p$ is a fixed
prime number and by a $G$-equivariant correspondence $X \rightsquigarrow V$ of degree
prime to $p$ we mean a dominant morphism $Y' \dasharrow Y$ of
degree prime to $p$ together with a $G$-equivariant rational map
$Y' \times_Y X$ to $V$. We say that $V$ is {\em versal} (respectively,
{\em $p$-versal}) if every dense open $G$-invariant subvariety
$V_0 \subset V$ is weakly versal (respectively, weakly $p$-versal).
Note that versality and $p$-versality are birational properties of
$V$, whereas weak versality and weak $p$-versality are not.

\begin{theorem}[Duncan--Reichstein]
 \label{ex.versal} Let $G$ be a finite $p$-group and 
  $X$ an irreducible $G$-variety over a base field $k$. If
  $X$ has a smooth $G$-fixed $k$-point, then $X$ is $p$-versal.
\end{theorem}
\begin{proof}
   See~\cite[Corollary 8.6]{duncan-reichstein}. \qed
\end{proof}

\begin{lemma} \label{lem.versal} Let $G$ be a finite group, $X$ an
  irreducible $G$-variety and $p$ a prime integer.

\smallskip
(a) If $X$ is versal, then $\ed X = \ed G$.

\smallskip
(b) If $X$ is $p$-versal, then $\ed(X; p) = \ed(G; p)$.
\end{lemma}

\begin{proof} (a) We need to show that $\ed X' \leqslant \ed X$ for
  every $G$-torsor $X' \to Y'$. Let $U$ be a $G$-invariant affine open
  subset of $X$ on which $G$ acts freely. By shrinking $U$ if
  necessary, we may find a morphism $\alpha: U/G \to Z$ and a
  $G$-torsor $W \to Z$ such that $U = (U/G) \times_Z W$.  Since $U$ is
  a $G$-invariant open subset of $X$, $U$ is weakly versal.  Thus
  there exists a $G$-equivariant rational map $ X' \dasharrow U$,
  equivalently a map $\beta: Y' \to U/G$ inducing an isomorphism
  $X' = Y' \times_{U/G} U$ . Composing $\beta$ and $\alpha$ we obtain
  a rational map $\gamma \colon Y' \dasharrow Z$ such that $X'$ is
  (generically) equal to $Y' \times_Z W$. We conclude that
  $\ed Y \leq \dim Z = \ed X$, as claimed.

Part (b) is proved by the same argument, with rational maps replaced by correspondences of degree prime to $p$.
\end{proof}

\begin{corollary} \label{cor.versal} Let $G$ be a finite group, $p$ a prime number, and $G_p$  a Sylow $p$-subgroup of $G$.
If $G_p$ is non-abelian, and $X$ is a $p$-versal $G$-variety, then $\dim X \geqslant p$.
\end{corollary}

\begin{proof} We have $\dim X  \stackrel{(i)}{\geqslant} \ed_G(X) 
\stackrel{(ii)}{\geqslant} \ed_{G_p}(X; p) \stackrel{(iii)}{=} \ed(G_p; p) \stackrel{iv)}{\geqslant} p$. 

Here (i) follows from the definition of essential dimension, (ii) from
Remark~\ref{rem.sbgrp}, (iii) from Lemma~\ref{lem.versal}(b), and (iv)
from \cite[Theorem 1.3]{meyer-reichstein-documenta}.
\end{proof}

Since $\dim\mathcal{A}_g=\dfrac{g(g+1)}{2}$, Corollary~\ref{cor.versal} allows us to see easily that the congruence
cover $\mathcal{A}_{g,pN}\to \mathcal{A}_{g,N}$ of the moduli space of principally polarized abelian varieties from the introduction is (usually) not $p$-versal.

\begin{corollary} \label{cor.non-versal} Fix $g \geqslant 2$. Then the
  $\Sp_{2g}(\mathbb{F}_p)$-torsor $\mathcal{A}_{pN}\to \mathcal{A}_N$
  from \S\ref{ss.ir} is not $p$-versal for any prime
  $p > \dfrac{g(g+1)}{2}$.
\end{corollary}

\begin{proof} By Corollary~\ref{cor.versal} it suffices to show that
  $G = \Sp_{2g}(\mathbb{F}_p)$ contains a non-abelian $p$-subgroup
  $H$. Since $G$ contains $\Sp_4(\mathbb{F}_p)$, we may assume without
  loss of generality that $g = 2$.  We may also assume that $G$ is the
  automorphism group of the symplectic form
  $x_1 \wedge x_4 + x_2 \wedge x_3$.  Let $H$ be the subgroup of
  lower-triangular matrices in $\Sp_4(\mathbb F_p)$. Clearly, $H$ is
  a $p$-group. An easy computation shows that the lower-triangular
  matrices
\[ a = \begin{pmatrix} 1 & 0 & 0 & 0 \\ 1 & 1 & 0 & 0 \\ 0 & 0 & 1 & 0 \\ 0 & 0 & -1 & 1 \end{pmatrix} 
\quad \text{and} \quad
b = \begin{pmatrix} 1 & 0 & 0 & 0 \\ 0 & 1 & 0 & 0 \\ 0 & 1 & 1 & 0 \\ 0 & 0 & 0 & 1 \end{pmatrix}
\] 
do not commute. Thus $H$ is a non-abelian subgroup of $G$, as claimed.
\end{proof}

\begin{remark}\label{r.benson-ad}  As mentioned in the introduction, 
  D.~Benson showed us a proof that, for $p>2$,  
$\ed_{\mathbb{C}}(\Sp_{2g}(\mb{F}_p);p)=p^{g-1}$.  With his permission,
we give it below in Theorem~\ref{t.Benson} of Appendix~\ref{a.benson}. 
\label{rem.non-versal} 
\end{remark}

\subsection{The fixed point method} 
Throughout this paper we will refer to the following result (and in particular to Proposition~\ref{prop.fixed-point}(b)) as ``the fixed point method.''

\begin{proposition} \label{prop.fixed-point} Let $G$ be a finite
  group, $X$  an irreducible generically free $G$-variety over a base field
  $k$ of characteristic $\neq p$, with $p$ a prime number.  Suppose
  $G$ has a subgroup $H$ such that $H$ is a $p$-group and $H$ has a
  smooth fixed point $X$. Then

\smallskip
(a) $\ed_G(X_0;p) \geqslant \ed_H(X_0; p) = \ed(H;p)$ for any dense open $G$-invariant subvariety $X_0 \subset X$.

\smallskip
(b) If $H=(\mathbb{Z}/p)^r$, then $\ed_G(X_0;p) \geqslant \ed_H(X_0; p) = r$.
\end{proposition}

In the case, where $\cha(k) = 0$,
Proposition~\ref{prop.fixed-point}(b) follows from~\cite[Theorem
7.7]{ry}, whose proof relies on equivariant resolution of
singularities.  The proof we present here does not use resolution of
singularities; in particular, it remains valid in prime
characteristic, as long as $\cha(k) \neq p$.

\begin{proof} (a) The inequality
  $\ed_G(X_0;p) \geqslant \ed_H(X_0; p)$ follows from
  Remark~\ref{rem.sbgrp}.  Since essential dimension at $p$ is a
  birational invariant of $H$-varieties,
  $\ed_H(X_0; p) = \ed_H(X; p)$.  Thus we may assume without loss of
  generality that $X_0 = X$. By Theorem~\ref{ex.versal}, $X$ is a
  $p$-versal $H$-variety.  Now Lemma~\ref{lem.versal} tells us that
  $\ed_H(X; p) = \ed(H; p)$.

  (b) By Remark~\ref{rem.elementary}, $\ed(H; p) = r$. The rest
  follows from part (a).
\end{proof}

\section{A criterion for the existence of fixed points}\label{s.crit}

In this section $k$ is an arbitrary algebraically closed field.

\begin{definition} \label{d.toric}
  $ $
  \begin{enumerate}
  \item A \emph{toroidal singularity} is a scheme $S$ over $k$
    together with an isomorphism of $S$ (which we suppress from the
    notation unless there is a possibility of confusion) with the
    spectrum of the completion of the local ring of a (normal) affine
    toric variety over $k$ at a torus fixed point.  We say that a
    toroidal singularity is \emph{simplicial} if the corresponding
    affine toric variety is simplicial.
  \item An action of a finite abelian group on a toroidal singularity
    is said to be toroidal if the action is induced via completion by
    the action of a finite subgroup of the torus on the corresponding
    toric variety.
  \item A \emph{toroidal map} of toroidal singularities is a morphism
    of schemes which is induced by completion from a toric morphism,
    (i.e., a morphism induced by a map of the corresponding
    semigroups) of the corresponding affine toric varieties.
  \end{enumerate}
\end{definition}

\begin{remark} \label{r.tor} The definitions (2) and (3) do depend on
  the choice of the toroidal structure, i.e., the isomorphism in (1).
\end{remark}

\begin{lemma} \label{l.sandwich} Let $\pi: T \to S$ be a finite
  surjective toroidal map of toroidal singularities. Then any normal
  scheme $S'$ with a finite map $\pi':S' \to S$ through which $\pi$
  factors is toroidal and the map $\pi'$ is also
  toroidal. Furthermore, if $S$ is simplicial then so is $S'$.
\end{lemma}

\begin{proof}
  By definition, there exist affine toric varieties $U_T$ and $U_S$
  and a toric morphism $p:U_T \to U_S$ inducing $\pi$ by completion at
  the torus fixed points.  The morphism $p$ corresponds to a map of
  lattices $p_M: M_T \to M_S$ together with rational polyhedral cones
  $C_T \subset M_T \otimes\mathbb{R} $,
  $C_S \subset M_S \otimes \mathbb{R}$, such that
  $p_M(C_T) \subset C_S$. Since $\pi$ is finite and surjective, $M_T$
  and $M_S$ must have the same rank and $p_M$ must be injective. We
  claim that we must also have $p_M(C_T) = C_S$: if this does not
  hold, then some positive dimensional face $F$ of $C_T$ must map to
  the interior of $C_S$. This implies that the closure in $U_T$ of the
  torus orbit corresponding to $F$, $Z_F$, maps to the torus fixed
  point of $U_S$. Since $Z_F$ contains the torus fixed point of $U_T$,
  it follows that $\pi$ cannot be finite, which gives a contradiction.
  
  From $p_M(C_T) = C_S$ it follows that $p$ is finite, and then by
  the Galois correspondence we see that $S'$ is the completion of the
  toric variety corresponding to the cone induced by $C_T$ in a
  sublattice $M_{S'}$ of $M_S$ containing $M_T$. The assertions of the
  lemma follow immediately from this.
  
\end{proof}

\begin{lemma} \label{l.abhyankar} Let $S$ be a simplicial toroidal
  singularity and $S^o \subset S$ the open set which is the complement
  of the completion of the boundary divisor. Let $\pi^o:T^o \to S^o$
  be a connected finite \'etale Galois cover of degree not divisible by
  $\car(k)$ and let $T$ be the normalisation of $S$ in $T^o$. Then $T$
  is a simplicial toroidal singularity and the induced map
  $\pi:T \to S$ is also toroidal.
\end{lemma}

\begin{proof}
  Since $S$ is simplicial, there exists a smooth toroidal singularity
  $S_1$ together with a finite toroidal map $\pi_1: S_1 \to S$
  satisfying $S_1 \times_S S^o = S_1^o$, where $S_1^o$ is the open
  complement of the boundary. Let $S_2^o$ be a connected component of
  $S_1^o \times_{S^o} T^o$. The projection $\pi_2^o$ to $S_1^o$ makes
  $S_2^o$ a finite \'etale Galois cover of $S_1$ of degree not dividing
  $\car(k)$. Let $S_2$ be the normalisation of $S_1$ in $S_2^o$ and
  let ${\pi}_2: S_2 \to S_1$ be the induced map.

  By Abhyankar's lemma \cite[Expose XIII, \S 5]{SGA1}, which is
  applicable since ${S}_1$ is complete and regular and
  ${S}_1 \bs S_1^o$ is a normal crossings divisor, there exists a
  formally smooth toroidal singularity ${S}_3$, a toroidal map
  ${\pi}_3: {S}_3 \to {S}_1$ and a map ${\pi}_{3,2}: {S}_3 \to {S}_2$
  such that ${\pi}_3 = {\pi}_2 \circ {\pi}_{3,2}$.  Lemma
  \ref{l.sandwich} now implies that ${S}_2$ and ${\pi}_2$ are both
  toroidal.

  The composition of two finite toroidal maps of toroidal
  singularities is again toroidal, so it follows that
  ${\pi}_1 \circ {\pi}_3$ is toroidal. By construction, there is a map
  ${\pi}_T: S_3 \to {T}$ such that
  ${\pi}_1 \circ {\pi}_3 = {\pi} \circ {\pi}_T$, so by Lemma
  \ref{l.sandwich} once again, we see that ${T}$ is a simplicial
  toroidal singularity and ${\pi}$ is also toroidal.
 
\end{proof}

The lemma below is a version of \cite[Proposition A.2]{ry}, which we
also use in the proof.

\begin{lemma} \label{l.down} %
  Let $A$ be a finite abelian group acting on a variety $U$ and
  $p: U \to V$ a proper $A$-equivariant map, with $A$ acting trivially
  on $V$. If there exists a toroidal singularity ${T}$ with a toroidal
  $A$-action and an $A$-equivariant rational map
  $h:{T} \dashrightarrow U$ such that the map
  $p \circ h: {T} \dashrightarrow V$ is a morphism, then $A$ has a
  fixed point in $U$.
\end{lemma}

\begin{proof}
  Let $C$ be an affine toric variety with a torus fixed point $c_0$
  such that ${T} \cong \widehat{\mathscr{O}}_{c_0,C}$.  Let
  $C^{\sharp}$ be the toric variety corresponding to the star
  subdivision of the cone corresponding to $C$ with respect to a ray
  in its interior. We have a birational proper morphism
  $C^{\sharp} \to C$ with the property that the fibre over $c_0$ is an
  irreducible divisor $E$.  Let
  $T^{\sharp} = {T} \times_{C}C^{\sharp}$; $T^{\sharp}$ is normal
  because the local rings at its closed points, which are the closed
  points of $E$, are the completions of the corresponding local rings
  on $C^{\sharp}$ which is normal (and excellent).  The map
  $C^{\sharp} \to C$ induces a proper birational morphism
  $\phi:T^{\sharp} \to {T}$ whose fibre over $t_0$, the closed
  point of ${T}$, is $E$. Since the action of $A$ on ${T}$ is
  induced from the torus action on $C$, $A$ acts on $T^{\sharp}$ and
  $E$ and the action on $E$ is toric.

  Let $\mathscr{O}_{E, T^{\sharp}}$ be the local ring of $E$ on
  $T^{\sharp}$. Since $E$ is a divisor and $T^{\sharp}$ is normal,
  $\mathscr{O}_{E, T^{\sharp}}$ is a dvr.  The map $h$ may be viewed
  as an $A$-equivariant rational map from $T^{\sharp}$ to $U$,
  so it gives an $A$-equivariant map
  $\spec(\fr(\mathscr{O}_{E, T^{\sharp}})) \to U$. It follows
  from the 
  assumptions that the composite of the maps
  \[
    \spec( \mathscr{O}_{E, T^{\sharp}}) \to {T}
    \stackrel{h}{\dashrightarrow} U \stackrel{p}{\to}
    V .
  \]
  is a morphism. By the valuative criterion of properness, the
  composed morphism $\spec(\mathscr{O}_{E, T^{\sharp}}) \to V$
  lifts to a morphism $\spec(\mathscr{O}_{E, T^{\sharp}}) \to U$
  which is $A$-equivariant, so we get an $A$-equivariant rational map
  from $E$ to $U$. Since $E$ lies over the closed point $t_0$ of
  ${T}$, this rational map factors through
  $Z = p^{-1}(\bar{f} \circ h(t_0)) \subset U$, which is
  a proper scheme. We now replace $E$ by any toric resolution
  $\tilde{E}$ of $E$, so we have an $A$-equivariant rational map from
  $\tilde{E}$ to $Z$. Since the $A$-action on $E$ comes from the
  torus and $\tilde{E}$ is proper, $E$ has a (smooth) $A$-fixed
  point. By \cite[Proposition A.2]{ry} it follows, that $Z$, hence
  also $U$, has an $A$-fixed point.
\end{proof}

\begin{proposition} \label{p.criterion} Let $Y$ be an irreducible
  variety over $k$ and $f:X \to Y$ an irreducible finite \'etale
  Galois cover with Galois group $G$. Let ${S}$ be a simplicial
  toroidal singularity and $S^o \subset {S}$ the complement of the
  boundary divisor. Suppose there exists a morphism $g: S^o \to Y$
  such that the image of the composite of
  $\pi_1^{\mathrm{et}}(S^o,s) \stackrel{g_*}{\to}
  \pi_1^{\mathrm{et}}(Y,g(s)) \twoheadrightarrow G$ \footnote{The
    second map is well defined only up to conjugation, but the claim
    does not depend on this choice.}  is a finite (abelian) group $A$
  of order not divisible by $\car(k)$ (for $s$ any geometric point of
  $S^o$). Assume that $g$ extends to a morphism
  $\bar{g}: {S} \to \ovl{Y}$, where $\ovl{Y} \supset Y$ is a partial
  compactification of $Y$. Then
  \begin{enumerate}
  \item Any $G$-equivariant partial compactification
    $\ovl{X} \supset X$ admitting a proper morphism
    $\bar{f}: \ovl{X} \to \ovl{Y}$ extending $f$ has an $A$-fixed
    point.
  \item Any smooth proper variety $X'$ with a $G$-action which is
    equivariantly birational to $X$ has an $A$-fixed point.
    \end{enumerate}
\end{proposition}

The statement and proof of this proposition have several elements in
common with \cite[\S 6]{cgr}.

\begin{proof}
  Consider the scheme $S^o \times_Y X$. It is finite \'etale over
  $S^o$ and has an action of $G$ induced by the action on $X$. The
  assumption on the map on fundamental groups implies that the
  connected components of $S \times_Y X$ are Galois covers of $S$ with
  Galois group $A$. Let $T^o$ be any one of these components and let
  $\pi^o:T^o \to S^o$ be the covering map.

  Let ${T}$ be the normalisation of ${S}$ in the function field of
  $T^o$, so ${T} \times_{{S}} S^o = T^o$. It follows from Lemma
  \ref{l.abhyankar} that ${T}$ and the induced map
  ${\pi}: {T} \to {S}$ are both toroidal. Since $A$ is the Galois
  group of the covering $\pi^o:T^o \to S^o$, the action of $A$ on
  ${T}$ is also toroidal.

  Let $h:{T} \dashrightarrow \ovl{X}$ be the $A$-equivariant rational
  map corresponding to the natural morphism from $T$ to $X$. The
  composed map $\bar{f} \circ h$ is equal to $\bar{g} \circ \pi$, so
  it is a morphism. We now apply Lemma \ref{l.down}, with
  $U = \ovl{X}$, $V = \ovl{Y}$, $p = \bar{f}$, to complete the proof
  of (1). Part (2) follows from (1) and ``going down''
  \cite[Proposition A.2]{ry}.
\end{proof}

\begin{remark}
  If $k = \mathbb{C}$ we may take $S$ to be the analytic germ of a
  toric variety at a torus fixed point, ${S^o}$ the complement of the
  toric boundary in $S$, and $g$, $\bar{g}$ to be complex analytic
  maps. This follows from Proposition \ref{p.criterion} by completing
  the local ring corresponding to $S$.
\end{remark}

\begin{theorem} \label{t.compress} Let $X$ be a smooth variety over an
  algebraically closed field $k$ with a free action of a finite group
  $G$. Let $A \subset G$ be an elementary abelian $p$-group of rank
  $r$ and let $Y = X/ G$.  Let ${S}$ be a simplicial toroidal
  singularity and $S^o \subset {S}$ the complement of the boundary
  divisor. Suppose there exists a morphism $g: S^o \to Y$ such that
  the image of the composite of
  $\pi_1^{\mathrm{et}}(S^o,s) \stackrel{g_*}{\to} \pi_1^{\mathrm{et}}(Y,g(s))$ is $A$.
  Furthermore, assume that $g$ extends to a morphism
  $\bar{g}: {S} \to \ovl{Y}$, where $\ovl{Y} \supset Y$ is a partial
  compactification of $Y$. If there exists a smooth partial
  compactification $\ovl{X} \supset X$ together with a proper morphism
  $\bar{f}: \ovl{X} \to \ovl{Y}$ extending $f$ then
  \[
    \ed_G(X;p) \geq \ed_A(X;p) = r .
  \]
\end{theorem}
Note that the assumption on the existence of $\ovl{X}$ is always
satisfied if $\car(k) = 0$.

\begin{proof}
  Since the $G$-action on $X$ is free, the quotient map $f: X \to Y$
  is finite \'etale. By Proposition \ref{p.criterion}, $\ovl{X}$ has an
  $A$-fixed point, so the theorem follows by applying Proposition
  \ref{prop.fixed-point} (taking $H$ there to be $A$).
\end{proof}

\section{The moduli space of curves}\label{s.tqft}

In this section we prove incompressibility results for two types of
covers of $\mathcal{M}_g$, first for covers that are pullbacks of
congruence covers of $\mathcal{A}_g$ and then for certain covers
arising from TQFTs. For congruence covers, our proof is characteristic
free and so extends the incompressibility over fields of
characteristic zero already proved in \cite{FKW19} to positive
characteristics; for the covers of the second type, the methods of
\emph{op.~cit.}~do not apply. Note that in \emph{op.~cit.},
incompressibility for $\mathcal{M}_g$ is deduced using the Torelli map
to $\mathcal{A}_g$, but our proof does not use this.

Both results are applications of Theorem \ref{t.compress}. In
\S\ref{s.monod} we make some monodromy computations needed for both
proofs, incompressibility for congruence covers is then proved in
\S\ref{s.congmg} and for the ``quantum'' covers in \S\ref{s.quantmg}.

For the basics of mapping class groups needed for this section the
reader may consult \cite{farb-margalit}.

\subsection{A monodromy computation} \label{s.monod}

\subsubsection{} \label{s.pants}

Let $g \geq 2$ and let $\Sigma$ be a closed oriented surface of genus
$g$. Let $\Mod(\Sigma)$ be the mapping class group of
$\Sigma$. Corresponding to any pants decomposition of $\Sigma$,
equivalently a collection $P$ of $3g-3$ mutually non-isotopic and
non-intersecting loops $\gamma_i$ in $\Sigma$, there is a free abelian
group $F_P(\Sigma) \subset \Mod(\Sigma)$ of rank $3g-3$ generated by
the Dehn twists around these loops.  Let $\Gamma_P$ be the dual graph
of the pants decomposition given by $P$: the vertices of $\Gamma_P$
are the connected components of $\Sigma \bs \cup_i \gamma_i$ and for
each $\gamma \in P$ there is an edge $e_{\gamma}$ joining the vertices
corresponding to the two components of which $\gamma$ is in the
boundary. This is a trivalent graph, possibly with loops and multiple
edges.

There is a canonical map $h_P:F_P(\Sigma) \to \Aut(H_1(\Sigma,
\mb{Z}))$ sending a diffeomorphism to its action on homology. This
map is not always injective; for example, the Dehn twist around a
separating loop acts trivially on homology. However, we have the following:

\begin{lemma} \label{l.maxmon} %
  For suitable choices of $P$ the map $h_P$ is injective; in fact, for
  any integer $N > 1$, the reduction of $h_P(F_P(\Sigma))$ in
  $\Aut(H_1(\Sigma, \mb{Z}/N\mb{Z}))$ is a free
  $\mb{Z}/N\mb{Z}$-module of rank $3g-3$.
\end{lemma} 

\begin{proof} 
  We explain the construction of one such $P$ by giving the dual graph
  as a graph embedded in the plane: Start with a convex $g$-gon $R$
  with edges labelled (consecutively) $e_1,e_2,\dots,e_g$. Inside $R$
  we choose $g-2$ distinct collinear points and connect them with
  edges $e_{g+1}, e_{g+2}, \dots, e_{2g-3}$. We then connect each of
  the vertices of $R$ with one of the chosen points in the interior by
  an edge in such a way that no edges intersect (except at vertices)
  and such that the resulting graph is trivalent. It is easy to see
  that this is always possible; the cases $g=2$ require some minor
  modifications which we leave to the reader. We call these edges
  $e_{2g-2}, e_{2g-1},\dots, e_{3g-3}$ and the resulting graph
  $\Gamma$. The edges of $\Gamma$ decompose the interior of $R$ into
  $g$ polygons with disjoint interiors; we label these
  $R_1, R_2,\dots, R_g$ with the edge $e_j$ lying in $R_j$,
  $j=1,2,\dots,g$.

  We construct a genus $g$ surface $\Sigma$ and a pants decomposition
  $P$ of $\Sigma$ such that $\Gamma \cong \Gamma_P$ by ``fattening''
  all the edges of $\Gamma$ with the loops $\gamma_i$ being circles
  (in $\mb{R}^3$) centered at a point on the interior of each edge
  $e_i$. A basis of $H_1(\Sigma, \mb{Z})$ is given as follows: let
  $A_j$, $j=1,2,\dots,g$ be the homology class of the circle
  $\gamma_j$. Each polygon $R_j$ also gives rise to a simple loop on
  $\Sigma$ whose homology class we call $B_j$. Note that we have not
  specified orientations of the various loops: for our purposes it
  suffices to arbitrarily fix one choice for each loop.

  We now compute the effect on homology of the Dehn twist
  $t_{\gamma_i}$ around the loop $\gamma_i$, starting with
  $i=1,2\dots, g$. By construction, we have
  $A_j\cdot B_k = \pm \delta_{j,k}$ it follows that we have:
\begin{align*}
  t_{\gamma_i}(A_j) &= A_j \ \forall j\\
  t_{\gamma_i}(B_j) &= B_j \ \forall j \neq i\\
  t_{\gamma_i}(B_i) &= B_i \pm A_i
\end{align*}
We now consider the Dehn twists around the loops $\gamma_i$ for
$i=g+1,\dots, 3g-3$. For such $i$ it is clear that
\begin{equation*}
  t_{\gamma_i}(A_j) = A_j \ \forall j.
\end{equation*}
The edge $e_i$ corresponding to the loop $\gamma_i$ lies on the
boundary of exactly two polygons $R_k$ and $R_l$, so
\[
  t_{\gamma_i}(B_j) = B_j \ \forall j \neq k,l.
\]
Finally, we have
\[
  t_{\gamma_i}(B_k) = B_k \pm A_k \pm A_l; \ t_{\gamma_i}(B_l) = B_l
  \pm A_l \pm A_k .
\]
The condition that the graph $\Gamma$ is trivalent implies that
distinct polygons $R_k$ and $R_l$ share at most one edge. From this
and the formulae above it follows that the images of the elements
$\{h_P(t_{\gamma_i})\}_{i=1}^{3g-3}$ in
$\Aut(H_1(\Sigma, \mb{Z}/N\mb{Z}))$ are linearly independent for all
$N>1$
\end{proof}

\subsubsection{} \label{s.versal}

Let $C_o$ be a totally degenerate stable curve of genus $g$ over
$\mathbb{C}$, and let $\nu:C \to B$ be the (analytic) versal
deformation of $C_o$, so $B$ is a ball of dimension $3g-3$ and
$C_o$ is the fibre over a point $o \in B$. Let $D \subset B$ be the
divisor over which $\nu$ is not smooth; this is a normal crossings
divisor with $3g-3$ irreducible components $D_i$. Let $B^o = B \bs D$,
and let $C_{B^o} = C \times_{B} B^o$, so $\nu: C_{B^o} \to B^o$ is a
family of smooth projective curves of genus $g$. Fixing a point
$u \in B^o$, we get a monodromy representation
$\rho:\pi_1(B^o,u) \to \Mod(C_u)$. The image of this homomorphism is a
free abelian group of rank $3g-3$; in fact, it is a subgroup of the
form $F_P(\Sigma)$ that we have considered above for a suitable
$P$. This can be seen as follows:

Choose a disc $B'$ in $B$ centered at the point $b_0$ corresponding to
$C_o$ and meeting each component $D_i$ of $D$ transversally and let
$C_{B'} = C \times_{B} B'$. Since $C_o$ is totally
degenerate it has $3g-3$ singular points $p_i$ each of which is an
ordinary double point. The divisors $D_i$ are in a natural bijection
with the singular points, $D_i$ being the locus of all points in $B$
over which the point $p_i$ ``remains singular''.  For each $p_i$
choose a small sphere $S_i$ in $C_{B^o}$ (in some fixed
metric) centered at $p_i$. Then if $b \in B' - \{b_0\}$ is sufficiently
close to $b_0$, $S_i \cap C_b$ is a simple loop $\gamma_i$
and these loops are pairwise disjoint (see, e.g., \cite[Expose
XIV]{SGA72}).  These $3g-3$ loops give rise to a pants decomposition
$P$ of $C_b$ (viewed as a differential manifold) and as a
topological space $C_o$ is obtained from $C_b$ by
contracting each $\gamma_i$ to a point.  The monodromy representation
then maps a loop $\delta_i$ in $B^o$ based at $b$ and going round the
divisor $D_i$ once to the Dehn twist (up to a sign depending on
orientations) corresponding to the loop $\gamma_i$ \cite[\S
2]{donaldson}.

\subsubsection{} \label{s.mon}

Any trivalent graph $\Gamma$ with $2g-2$ vertices gives rise to a
unique totally degenerate stable curve $C_{\Gamma}$ of genus $g$ over
any field $k$. We choose a copy of $\mb{P}_k^1$ with three marked
rational points for each vertex of $\Gamma$ and label the marked
points with the edges incident on the vertex, a loop being counted
twice.  We glue these copies of $\mb{P}_k^1$ along the marked points
by identifying pairs of points marked by the same edge of
$\Gamma$. The pants decomposition $P$ that one gets by taking $C_o$ to
be $C_{\Gamma}$ has $\Gamma_P \cong \Gamma$.

For an arbitrary field $k$ and a totally degenerate stable curve $C_o$
over $k$, we have a versal deformation $\nu:C \to B$ of $C_o$, where
now $B = \spec k [[x_1,x_2,\dots,x_{3g-3}]]$. As above, $\nu$ is
smooth outside a normal crossings divisor $D \subset B$ (which we may
take to be the zero set of $x_1x_2,\dots x_{3g-3}$) and we let
$B^o = B \backslash D$. Let $u$ be a geometric point of
$B^o$. Although the mapping class group is no longer meaningful over
$k$, we still have a monodromy representation
$\rho:\pi_1^{\mathrm{et}}(B^o, u) \to \aut H^1_{\mathrm{et}}(C_u,
\mb{Z}/N\mb{Z})$, where $(N, \car(k)) = 1$.

\begin{lemma} \label{l.mon} %
  Let $k$ be an algebraically closed field with $(N, \car(k)) = 1$. If
  the dual graph of $C_o$ is as in Lemma \ref{l.maxmon}, then the
  image of the monodromy representation $\rho$ is a free
  $\mb{Z}/N\mb{Z}$-module of rank $3g-3$.
\end{lemma}

\begin{proof}
  If $k = \mb{C}$, then the comparison theorem for \'etale
  fundamental groups and \'etale cohomology shows that in this case the
  Lemma is implies by Lemma \ref{l.maxmon} and the discussion in \S
  \ref{s.versal}. This implies the lemma for $k$ of characteristic
  zero by Abhyankar's lemma.

  If $\car(k) > 0$ we let $W(k)$ be the ring of Witt vectors of $k$
  and let $\mc{B}$ be the versal deformation of $C_o$ over
  $W(k)$. So $\mc{B} = \spec W(k)[[x_2,x_2,\dots,x_{3g-3}]]$ and
  the universal family $\nu: \mc{C} \to \mc{B}$ is smooth over
  the complement $\mc{B}^o$ of a relative normal crossings divisor
  $\mc{D} \subset \mc{B}$. The lemma then follows from the
  case $\car(k) = 0$ using the structure of the tame fundamental group
  of $\mc{B}^o$ (\cite[Expose XIII, Corollaire 5.3]{SGA1}).
\end{proof}

\subsection{Incompressibility for congruence covers of 
\texorpdfstring{$\mathcal{M}_g$}{Mg}}
\label{s.congmg}

In this section $k$ is an arbitrary algebraically closed field.

\subsubsection{}

Let $\mathcal{M}_g$ be the moduli space of smooth projective curves of
genus $g \geq 2$ over $k$. Let $\ovl{\mathcal{M}}_g$ be the moduli
space of stable curves; this is an irreducible normal projective
variety. For an integer $N>2$ not divisible by $\car(k)$, let
$\mathcal{M}_g[N]$ be the moduli space of smooth curves with
(symplectic) level $N$-structure. Although $\mathcal{M}_g$ is not
smooth, $\mathcal{M}_g[N]$ is an irreducible and smooth variety. Let
$\ovl{\mathcal{M}}_g[N]$ be the normalisation of $\ovl{\mathcal{M}}_g$
in $\mathcal{M}_g[N]$; it is a normal projective variety but it is not
smooth.

\begin{lemma} \label{l.degsmooth} %
  Let $C_o$ be a totally degenerate curve of genus $g$ over $k$ whose
  dual graph $\Gamma$ is of the form described in Lemma
  \ref{l.maxmon}. If $N> 2$ then $\ovl{\mathcal{M}}_g[N]$ is smooth at
  all points lying above the point corresponding to $C_o$ in
  $\ovl{\mathcal{M}}_g$.
\end{lemma}

\begin{proof}
  This follows immediately from \cite[Satz II]{mostafa}, at least if
  $\car(k) = 0$. We give a self-contained proof below that also works
  if $\car(k) > 0$.

  Let $B$, $C$, etc., be as in \S \ref{s.mon}. Let $B_0[N]$ be the
  $\mf{Sp}(2g, \mb{Z}/N\mb{Z})$-torsor over $B_0$ given by adding
  (symplectic) level $N$ structure and let $B[N]$ be the normalisation
  of $B$ in $B_0[N]$. Let $B'$ be any connected component of
  $B[N]$. It follows from Lemma \ref{l.mon} that the map $B'$ is given
  by extracting $N$-th roots of all the $x_i$, so
  $B' \cong \spec k[[y_1,y_2,\dots,y_{3g-3}]]$ with $y_i^N = x_i$. In
  particular, $B'$ is regular.

  The finite group $\aut(C_o)$ acts on all the objects above. It acts
  faithfully on $H^1_{\mathrm{et}}(C_o, \mb{Z}/ N\mb{Z})$ which is identified
  with the invariants of the monodromy on
  $H^1_{\mathrm{et}}(C_u, \mb{Z}/ N\mb{Z})$. It follows that $B'$ has trivial
  stabilizer in $\aut(C_o)$.

  We have a natural map $\iota:B_0[N] \to \mathcal{M}_g[N]$ since
  $\mathcal{M}_g[N]$ is a fine moduli space and this extends (by
  normality and the finiteness of the morphism $B[N] \to B$) to a
  morphism $\bar{\iota}: B[N] \to \ovl{\mathcal{M}}_g[N]$.  Since the
  map $B \to \ovl{\mathcal{M}}_g$ induces an isomorphism of
  $B/\aut(C_o)$ with the formal neighbourhood of
  $[C_0] \in \ovl{\mathcal{M}}_g$, the fact that $B'$ has trivial
  stabilizer in $\aut(C_o)$ implies that $\bar{\iota}$ identifies $B'$
  with the formal neighbourhood of
  $\bar{\iota}(o') \in \ovl{\mathcal{M}}_g[N]$, where $o'$ is the
  closed point of $B'$. The lemma follows from this since $B'$ is
  regular.

\end{proof}

\begin{remark}
  We only need Lemma \ref{l.degsmooth} for the $g$ for which there
  is no graph $\Gamma$ as in Lemma \ref{l.maxmon} such that
  $\aut(\Gamma)$ is trivial. It can be easily checked that such graphs
  exist for all $g \geq 7$.
\end{remark}

As above, we continue to assume that $\car(k) \nmid N$. Suppose $p$ is
a prime such that $p \nmid \car(k)N$. Then forgetting the level $p$
structure induces a finite \'etale morphism
$f: \mc{M}_g[pN] \to \mc{M}_g[N]$ making $ \mc{M}_g[pN]$ into an
$\Sp_{2g}(\mb{F}_p)$-torsor over $\mc{M}_g[N]$.

The following theorem gives a new proof of \cite[Corollary 4]{FKW19}
and also extends it to fields of positive characteristic.

\begin{theorem} \label{t.incompmg} %
  $ \mc{M}_g[pN]$ is $p$-incompressible as an
  $\Sp_{2g}(\mb{F}_p)$-torsor. Furthermore, for $g>2$, and any
  $n > 1$ such that $p \mid n$ and $\car(k) \nmid n$, the map
  $\mc{M}_g[n] \to \mc{M}_g$ is $p$-incompressible.
\end{theorem}
  
\begin{proof}
  Let $C_o$ be a totally degenerate stable curve of genus $g$ with
  dual graph as in Lemma \ref{l.maxmon}.  Let $B$ be the versal
  deformation space of $C_o$ as in \S \ref{s.mon}. Let $S^0$ be a
  connected component of the \'etale cover of $B^o$ given by
  trivialising the finite local system $R^1 \nu_* (\mb{Z}/N\mb{Z})$
  and let $S$ be the normalisation of $B$ in $S^0$. Then $S$ is a
  simplicial toroidal singularity since $\car(k) \nmid N$; it is in
  fact smooth, but we will not need this. The pullback of $C$ to $S^0$
  induces a morphism $g:S^0 \to \mc{M}_g[N]$ and since $p \nmid N$, by
  Lemma \ref{l.mon} for any geometric point $s$ of $S^0$ the image of
  the composite
  $\pi_1^{\mathrm{et}}(S^0, s) \stackrel{g_*}{\longrightarrow}
  \pi_1^{\mathrm{et}}(\mathcal{M}_g[N], g(s)) \twoheadrightarrow
  \Sp_{2g}(\mb{F}_p)$ is an elementary abelian $p$-group of rank
  $3g-3 = \dim(\mathcal{M}_g[N])$.

  By Lemma \ref{l.degsmooth} there exists a nonsingular Zariski open subset
  $U$ (resp.~$V$) of $\ovl{\mathcal{M}}_g[pN]$
  (resp.~$\ovl{\mathcal{M}}_g[N]$) containing $\mc{M}_g[pN]$
  (resp.~$\mc{M}_g[N]$) and all the points parametrising curves
  isomorphic to $C_o$ such that the map
  $f: \mc{M}_g[pN] \to \mc{M}_g[N]$ extends to a finite morphism
  $\bar{f}: U \to V$. We now set $X = \mathcal{M}_g[pN]$,
  $Y = \mathcal{M}_g[N]$, $\ovl{X} = U$, $\ovl{Y} = V$ and
  $f:X \to Y$, $\bar{f}: \ovl{X} \to \ovl{Y}$ as above. The first part
  of the theorem will follow from Theorem \ref{t.compress} if we can
  show that the map $g:S^0 \to Y$ extends to a morphism
  $\bar{g}: S \to \ovl{Y}$

  To see this, we note that the universal family of stable curves over
  $B$ induces a morphism $h: B \to \ovl{\mathcal{M}}_g$. Let
  $B' = B \times_{\ovl{\mathcal{M}}_g} \ovl{\mathcal{M}}_g[N]$. $B'$ is finite
  over $B$ and the map $g:S^0 \to Y$ clearly factors through
  $B'$. Since $S$ is normal and finite over $B$, the map $S^0 \to B'$
  extends to a morphism $S \to B'$ which gives an extension of $g$ as
  a map $g': S \to \ovl{\mathcal{M}}_g[N]$. By construction, the
  closed point of $S$ maps to a point in $\ovl{X}$, so it follows that
  $g$ extends to a morphism $\bar{g}: S \to \ovl{Y}$ as desired.

  To prove the second statement, we may assume that $n=p$. Choose any
  integer $N> 2$ such that $(N, p\car(k)) = 1$. Since $g >2$, the
  generic curve of genus $g$ has no nontrivial automorphisms, so $\mc{M}_g[p]
  \times_{\mc{M}_g}\mc{M}_g[N]$ is equal to $\mc{M}_g[pN]$ over the
  generic point of $\mc{M}_g[N]$. We complete the proof by applying
  the first part of the theorem.
\end{proof}

\begin{remark} \label{r.mg}%
  The second part of Theorem \ref{t.incompmg} for
  $g=2$ follows from the corresponding statement for $\mc{A}_g$ proved
  in Corollary \ref{cor.main} and Theorem \ref{c.agp}.
\end{remark}

\subsection{Incompressibility for some ``quantum'' covers of 
\texorpdfstring{$\mathcal{M}_g$}{Mg}}
\label{s.quantmg}

We now apply Proposition~\ref{p.criterion} to prove incompressibility
for certain covers of moduli spaces of curves arising from certain
($3$-dimensional) Topological Quantum Field Theories. Although our
method is quite general, for the sake of concreteness we restrict
ourselves to the TQFTs constructed explicitly in \cite{bhmv}.

Throughout this section the base field is $\mathbb{C}$.

\subsubsection{}
We first recall some of the properties of the TQFTs that we shall
need.  Let $g \geq 2$ and let $\Sigma$ be a closed oriented surface of
genus $g$. The three dimensional TQFTs constructed in \cite{bhmv},
which are indexed by integers $p \geq 1$, give rise to a projective
representation $\rho_p$ of the mapping class group $\Mod(\Sigma)$
in a vector space $V_p(\Sigma)$ (over the cyclotomic field
$\mb{Q}(\mathrm{exp}({\tfrac{2\pi i}{4p}}))$).

We now assume that $p$ is odd for simplicity: this is not a
significant restriction since the $V_p(\Sigma)$ for even $p$ can be
expressed as a tensor product of $V_p(\Sigma)$ for odd $p$ and another
simple representation (depending only on $g$).  For each $P$ as in
\S\ref{s.pants} there is a natural basis $\{u_{\sigma}\}$ of
$V_p(\Sigma)$ in which the action of $F_P(\Sigma)$ is diagonalised
\cite[Theorem 4.11]{bhmv}. The basis is parametrised by
\emph{admissible colorings} of $\Gamma_P$ \cite[Definition 4.5]{bhmv}
which are maps
\[
  \sigma:E(\Gamma_P) \to \{0,2,\dots, p-2\}
\]
satisfying the following admissibility condition: for all vertices $v$
of $\Gamma_P$, if $e_1$, $e_2$ and $e_3$ are the three edges incident
on $v$ (for a loop the same edge is repeated twice) then
\begin{equation} \label{e.triang}
  |\sigma(e_1) - \sigma(e_2)| \leq \sigma(e_3) \leq \sigma(e_1) +
  \sigma(e_2) .
\end{equation}
The Dehn twist corresponding to an edge $\gamma$ acts on $u_{\sigma}$
by multiplication by a $p$-th root of unity depending only on the
colour $\sigma(e_{\gamma})$: if $A= \mathrm{exp}(\tfrac{2\pi i}{4p})$
then for any colour $i$ this is $(-1)^iA^{i^2 + 2i}$ \cite[\S 5.8,
Remark 7.6 (ii)]{bhmv}. One sees from the formula that if $p$ is an
odd prime, then these roots of unity are distinct $p$-th roots of
unity for distinct even colours.

\begin{lemma} \label{l.rank} %{\color{blue} This can probably be
                             %improved}
  The groups $\rho_p(F_P(\Sigma))$ are elementary abelian $p$-groups
  of rank $3g-3$ if $p>5$ is a prime.
\end{lemma}

\begin{proof}
  
  Let $\gamma \in P$. If $\gamma$ is a loop, let $\sigma_{\gamma}$ be
  the coloring of $\Gamma_P$ given by
\begin{equation}
  \sigma_{\gamma}(e_{\gamma'}) = \begin{cases}
    2 \mbox{ if } \gamma' = \gamma \\
    0 \mbox{ if } \gamma' \neq \gamma .
  \end{cases}
\end{equation}
If $\gamma$ is not a loop let $\sigma_{\gamma}$ be the coloring of
$\Gamma_P$ given by
\begin{equation}
  \sigma_{\gamma}(e_{\gamma'}) = \begin{cases}
    4 \mbox{ if } \gamma' = \gamma \\
    2 \mbox{ if } \gamma' \neq \gamma .
  \end{cases}
\end{equation}
One easily sees that $\sigma_{\gamma}$ is an admissible coloring in
both cases. Moreover, the Dehn twist $t_{\gamma}$ around $\gamma$ acts
by a non-trivial $p$-th root of unity on $u_{\sigma_{\gamma}}$ in both
cases, whereas the other Dehn twists $t_{\gamma'}$ all act by a
different root of unity independent if $\gamma \neq \gamma'$.

Let $\sigma_0$ (resp.~$\sigma_2$) be the coloring which is the
constant function $0$ (resp.~$2$). Suppose
$\prod_{\gamma \in P}t_{\gamma}^{a_{\gamma}}$ is in
$\ker(\rho_p(F_P(\Sigma)))$. If $\gamma$ is a loop then by considering
the restriction of the representation to the subspace spanned by
$u_{\sigma_0}$ and $u_{\sigma_{\gamma}}$ we see immediately that
$a_{\gamma}$ must be a multiple of $p$. If $\gamma$ is not a loop we
get the same conclusion by using the subspace spanned by
$u_{\sigma_2}$ and $u_{\sigma_{\gamma}}$. Thus, $\rho_p(F_P(\Sigma))$
is is an elementary abelian $p$-group of rank equal to $\# P = 3g-3$.
Since $F_P(\Sigma)$ acts trivially on $u_{\sigma_0}$ we conclude that
the same holds for the projective image in  $PGL(V_p(\Sigma))$.
\end{proof}

\subsubsection{}

Since $\Mod(\Sigma)$ is finitely generated, the representations
$\rho_p$ (for any $p \geq 1$) can be reduced modulo all but finitely
many primes $q$ of the subring of $\mb{C}$ generated by the matrix
coefficients, giving rise to an infinite sequence of finite quotients
$G^{p,q}$ of $\Mod(\Sigma)$.  In general, these quotients might depend
on the choice of lattice used to define the reduction of the
representations, but for $p$ an odd prime the explicit lattices in
$V_p(\Sigma)$ invariant under the action of $\Mod(\Sigma)$ that were
constructed in \cite{gil-mas} can be used to uniquely specify the
$G^{p,q}$ (for all $q$).

The group $\Mod(\Sigma)$ can be identified (after choosing a base
point) with the fundamental group of the moduli \emph{stack}
$\mathbf{M}_g$ of smooth projective curves of genus $g$ over
$\mb{C}$. Therefore, the finite quotients $G^{p,q}$ of $\Mod(\Sigma)$
give rise to \'etale covers of the stack $\mathbf{M}_g$, which we
denote by $\mathbf{M}_g^{p,q}$.

\begin{lemma} \label{l.disjoint} Suppose $p> 3$ is a prime and
  $n\neq p$ is also a prime. Let $\Gamma^{p,q}$ be the kernel of the
  surjection $\Mod(\Sigma) \to G^{p,q}$ and let $\Gamma_n$ be the
  kernel of the surjection
  $\Mod(\Sigma) \to \Aut(H_1(\Sigma, \mb{Z}/n\mb{Z}))$. Then
  $\Gamma^{p,q} \cdot \Gamma_n = \Mod(\Sigma)$.
\end{lemma}

\begin{proof}
  From the proof of Lemma \ref{l.rank} we see that the Dehn twists
  $\delta$ around any of the loops $\gamma$ maps to an element of
  order $p$ under $\rho_p$, so $\delta^p \in \Gamma^{p,q}$. Since
  $p\neq n$, if $\gamma$ is a non-separating loop then $\delta^p$ is
  an element of exact order $n$ in
  $\Aut(H_1(\Sigma, \mb{Z}/n\mb{Z}))$. Since the only normal
  subgroup of this group is its centre, it follows that $\Gamma^{p,q}$
  surjects onto $\Aut(H_1(\Sigma, \mb{Z}/n\mb{Z}))$.
\end{proof}

For any integer $N > 2$, let $\mathbf{M}_g[N]$ be the moduli stack of
smooth projective curves of genus $g$ with level $n$ structure. It is
in fact equal to its coarse moduli space, the variety
$\mathcal{M}_g[N]$ from \S \ref{s.congmg}.  We set
$\mc{M}_g^{p,q}[N]:= \mathbf{M}_g^{p,q}
\times_{\mathbf{M}_g}\mc{M}_g[N]$. By Lemma \ref{l.disjoint},
$\mathcal{M}_g^{p,q}[N]$ is irreducible and the projection map
$\pi_g^{p,q}:\mathcal{M}_g^{p,q}[N] \to \mathcal{M}_g[N]$ is a finite
etale $G^{p,q}$-cover.

\begin{theorem} \label{t.quantum} The $G^{p,q}$-covers
  $\pi_g^{p,q}:\mathcal{M}_g^{p,q}[N] \to \mathcal{M}_g[N]$ are
  $p$-incompressible for $p>5$, $N> 2$ and $p \nmid N$ (and all but
  finitely many primes $q$ of the relevant cyclotomic field).
\end{theorem}

\begin{remark} \label{r.g>2}
  $ $
  \begin{enumerate}
  \item If $g>2$ then $\mc{M}_g$ is generically a variety, so we
    have a finite $G^{p,q}$-cover $\pi_g^{p,q}:\mathcal{M}_g^{p,q} \to \mathcal{M}_g$
    (which is only generically \'etale). Theorem \ref{t.quantum} implies
    that this cover is generically incompressible.
  \item If $g>3$ there exists a totally degenerate stable curve with
    trivial automorphism group. In this case the proof below may be
    carried out without the level $N$-structure.
    \end{enumerate}
  \end{remark}

\begin{proof}
  Let $C_o$ be a totally degenerate stable curve of genus $g$, and let
  $\nu:C \to B$, $D$, $B$, etc. be as in \S \ref{s.versal}.  By the
  discussion there and Lemma \ref{l.rank}, the image in $G^{p,q}$ of
  $F_P(\Sigma)$ (identified with $\pi_1(B^o,b)$) is an elementary
  abelian $p$-group of rank $3g-3$. On the other hand, the image of
  the monodromy representation in $\Aut(H_1(C_b, \mb{Z}/N\mb{Z}))$ is
  a finite abelian group of exponent $N$ (since the mondromy around
  each irreducible component of $D$ is unipotent).

  Let $S^o$ be a connected component of
  $B^o \times_{\mathbf{M}_g} \mathcal{M}_g[N]$, where the map to $\mathbf{M}_g$
  is the classifying map. The map $S^o \to B^o$ is an abelian cover of
  exponent $N$. Let $S$ be the normalisation of $B$ in $S^o$, so this
  is a (simplicial) toroidal singularity and $S^o \subset {S}$ is the
  complement of the boundary divisor. Since $p \nmid N$, for
  $s \in S^o$ the image of the monodromy representation
  $\pi_1(S^o,s) \to \pi_1(\mathcal{M}_g[N]) \to G^{p,q}$ is still an
  elementary abelian $p$-group of rank $3g-3$.

  We conclude by applying Theorem \ref{t.compress} with
  $Y = \mathcal{M}_g[N]$, $X = \mathcal{M}_g^{p,q}[N]$,
  $f =\pi_q^{p,q}$, $\ovl{Y}= \overline{\mc{M}}_g[N]$, $\ovl{X}$ any
  equivariant smooth compactification of $ \mathcal{M}_g^{p,q}[N]$
  which has a map $\bar{f}$ to $\overline{\mc{M}}_g[N]$ extending $f$
  (which exists because $k = \mb{C}$), $S^o$ and $S$ as above and $g$
  and $\bar{g}$ the natural classifying maps.

\end{proof}

\begin{remark} \label{r.mr} By a theorem of Masbaum and Reid
  \cite{mas-reid}, for each $g \geq 2$ there are infinitely many
  integers $N$, and for each such $N$ an infinite set of (rational)
  primes $P_N$, such that the group $PSL(N,r)$, for $r \in P_N$,
  occurs as one of the groups $G^{p,q}$. In contrast to this, it
  follows from the congruence subgroup property of
  $\mathbf{Sp}_{2g} (\mb{Z})$, $g \geq 2$, that there are no such
  covers of $\mc{A}_g$ for $g \geq 2$ if $N > 2g$.
\end{remark}

\section{Locally symmetric varieties}\label{s.sv}

In this section, we first (in \S\ref{lss}) fix some notation involving
algebraic groups and Hermitian symmetric domains and state our main
theorems regarding locally symmetric varieties:
Theorems~\ref{t.main-pure} and Theorem~\ref{thm.main2}.
Theorem~\ref{t.main-pure} is then proved in \S\ref{s.gherm} and
Theorem~\ref{thm.main2} is proved in \S\ref{s.proof-of-thm2}.

Our basic references for Hermitian symmetric domains are \cite{helga}
(where they are called Hermitian symmetric spaces of noncompact type)
and \cite{amrt}.

\subsection{Locally symmetric varieties, boundary components and 
  incompressibility}\label{lss}

We say that a semisimple algebraic group $\mathbf{G}$ over
$\mathbb{R}$ is of \emph{Hermitian type} if the Lie group
$G := \mathbf{G}^{\ad}(\mathbb{R})_+$ is isomorphic to $\Aut(D)^0$, the
identity component of the group of biholomorphisms of a Hermitian
symmetric domain $D$.  We note that $G$ and $D$ determine each other,
$D$ being biholomorphic to $G/K$, where $K$ is a maximal compact
subgroup of $G$ and $G/K$ having its natural $G$-invariant complex
structure. The choice of isomorphism does not matter for our purposes,
we fix one and then identify $D$ with $G/K$.  A semisimple algebraic
group over $\mathbb{Q}$ is of \emph{Hermitian type} if
$\mathbf{G}_{\mathbb{R}}$ is of Hermitian type.

Suppose $\mathbf{G} = \mathbf{G}^{\ad}$ and $\Gamma$ is an arithmetic
subgroup of $\mathbf{G}(\mathbb{Q})$ contained in $G$, then, by
Baily--Borel~\cite{BailyBorel}, the complex-analytic space
$\Gamma\backslash D$ has the structure of a quasi-projective complex
variety.  Moreover, if $\Gamma$ is neat, then this structure is unique
by Borel's Extension Theorem~\cite{BorelExt}.  (See \S\ref{s.AC} for a
review of the notions of arithmetic and congruence subgroups.) In
other words, there is a uniquely defined quasi-projective variety
$M_{\Gamma}$ whose analytification $M_{\Gamma}^{\mathrm{an}}$ is
$\Gamma\backslash D$.  The varieties $M_{\Gamma}$ will be called
\emph{locally symmetric varieties}.  We often abuse notation and write
$\Gamma \backslash D$ for $M_{\Gamma}$.

Now suppose that $\Delta$ is a finite index subgroup of $\Gamma$.  We
then have a morphism $\Delta\backslash D\to \Gamma\backslash D$ of
complex analytic spaces and, again using~\cite{BorelExt}, we get a
corresponding morphism $\pi: M_{\Delta}\to M_{\Gamma}$ of
quasi-projective varieties.  Moreover, if we assume that $\Gamma$ is
neat and $\Delta\unlhd \Gamma$, then $M_{\Gamma}$ is smooth and $\pi$
is a $\Gamma/\Delta$-torsor.

We have $D=G/K$ where $K$ is a maximal compact subgroup of $G$.
Moreover, the Cartan involution $\sigma$ corresponding to $K$ gives us
a splitting $\mathfrak{g}=\mathfrak{k}\oplus\mathfrak{p}$, where
$\mathfrak{g}$ and $\mathfrak{k}$ are the Lie algebras of $G$ and $K$
respectively.  Here $\mathfrak{k}$ and $\mathfrak{p}$ are the $+1$ and
$-1$ eigenspaces of $\sigma$ respectively.  The space $\mathfrak{p}$,
which can be identified with the tangent space of $D$ at the point
corresponding to $K$, has a complex structure $J$.  This gives rise to
a decomposition
$\mathfrak{p}_{\mathbb{C}}=\mathfrak{p}_{+}\oplus \mathfrak{p}_{-}$,
where $\mathfrak{p}_{\pm}$ is the $\pm i$-eigenspace of $J$.  The
Harish-Chandra embedding theorem~\cite[Theorem 2.1]{amrt}, then gives
a holomorphic embedding of $D$ into $\mathfrak{p}_+$.

Write $\ovl D$ for the closure of $D$ in $\mathfrak{p}_+$.  Then the
action of $G$ on $D$ extends to $\ovl D$, which is itself a union of
Hermitian symmetric domains called \emph{boundary
  components}~\cite[Definition 3.2]{amrt}.  The \emph{normalizer} of a
boundary component $F$ is then $N(F):=\{g\in G: gF=F\}$.  (Note that
the space $D$ itself is a boundary component with $N(F)=G$.)  If we
write $G$ as a product $G_1\times\cdots \times G_k$ with $G_i$ simple,
then the association $F\mapsto N(F)$ sets up a one-one correspondence
between boundary components and parabolic subgroups $P\leq G$ of the form
$P_1\times \cdots\times P_k$ with each $P_i$ containing a maximal
parabolic subgroup of $G_i$.  (So  each $P_i$ is either maximal parabolic
in $G_i$ or equal to $G_i$ itself.) Let $W(F)$ be the unipotent
radical of $N(F)$ and $U(F)$ the centre of $W(F)$.

A boundary component is said to be \emph{rational} if
$N(F)=\mathbf{P}(\mathbb{R})\cap G$ where $\mathbf{P}$ is a rational
parabolic subgroup of $\mathbf{G}$.  We will call such a parabolic
subgroup \emph{$d$-cuspidal} if the corresponding boundary component
is $d$-dimensional. The Baily--Borel % and the
% Ash--Mumford--Rapoport--Tai (AMRT)
compactfication of $\Gamma\backslash D$ involves adding strata
corresponding to rational boundary components~\cite{BailyBorel}, the
stratum associated to a $d$-cuspidal parabolic being of dimension
$d$. (When $D$ has no rational boundary components the space
$\Gamma\backslash D$ is compact for any arithmetic subgroup
$\Gamma\leq\mathbf{G}(\mathbb{Q})\cap G(\mathbb{R})$.)

With the above notation, we can state a theorem. 

\begin{theorem}\label{t.main-pure}
  Let $\mathbf{G}$ be an algebraic group over $\mathbb{Q}$ of
  Hermitian type with corresponding Hermitian symmetric domain $D$,
  let $F$ be a rational boundary component of $D$ and set
  $U:=U(F)$. Let $\Delta$ and $\Gamma$ be neat arithmetic subgroups of
  $G \cap \mathbf{G}(\mathbb{Q})$ with $\Delta\unlhd \Gamma$, and set
  $\mathbb{U}:=(U\cap \Gamma)/(U \cap \Delta)$.  Then any smooth
  $\mathbb{U}$-equivariant compactification $\overline{M}_{\Delta}$ of
  $M_{\Delta}$ has a $\mathbb{U}$-fixed point.
\end{theorem}

Coupled with the fixed point method, Theorem~\ref{t.main-pure} can be
used to give lower bounds for the essential dimension of congruence
covers. Tube domains, see \S\ref{s.proof-of-thm2}, are a special class
of Hermitian symmetric domains, and if $D$ is such and $\mathbf{P}$ is
$0$-cuspidal then we can even prove incompressibility for certain
covers. Since the proof in this case is very simple and does not use
the general theory of Hermitian symmetric domains, we state and prove
this special case as Theorem~\ref{thm.main2} below.

(See \S\ref{s.AC} for the notion of a subgroup defined by congruence
conditions, which we use in the statement.)

\begin{theorem} \label{thm.main2} Let $U$ be a finite dimensional real
  vector space and let $D= C + iU$ be a tube domain in
  $U_{\mathbb{C}}$ associated to a self-adjoint homogenous open cone
  $C \subset U$.  Let $\mathbf{G}$ be an adjoint algebraic group over
  $\mathbb{Q}$ such that
  $G :=\mathbf{G}(\mathbb{R})_+ \cong \Aut(D)^0$.  Let $P$ be the
  normalizer of $U$ in $G$ ($U$ acts on $D$ by translations) and
  assume that $P = \mathbf{P}(\mathbb{R}) \cap G$, where
  $\mathbf{P} \subset \mathbf{G}$ is a parabolic subgroup.  Let
  $\Delta$ and $\Gamma$ be neat arithmetic subgroups of
  $G \cap \mathbf{G}(\mathbb{Q})$ with $\Delta\unlhd \Gamma$, and set
  $\mathbb{U}:=(U\cap \Gamma)/(U \cap \Delta)$.  Then:
  \begin{enumerate}
  \item Any smooth $\mathbb{U}$-equivariant compactification
    $\overline{M}_{\Delta}$ of $M_{\Delta}$ has a $\mathbb{U}$-fixed
    point.
  \item If the rank of $\mathbb{U} \otimes \mathbb{F}_p$ is equal to
    $\dim(D)$ then the cover $\pi: M_{\Delta} \to M_{\Gamma}$ is
    $p$-incompressible.
    \end{enumerate}
    Furthermore, given $\Gamma$ as above, for any prime $p$ there
    exists $\Delta \unlhd \Gamma$ such that the rank of
    $\mathbb{U} \otimes \mathbb{F}_p$ is equal to $\dim(D)$.
\end{theorem}
A proof of Theorem~\ref{thm.main2} will be given in
Section~\ref{s.proof-of-thm2}.

\begin{remark} \label{r.nonneat}%
  One can sometimes get essentially the same conclusions as in
  Theorems \ref{t.main-pure} and \ref{thm.main2} even when $\Gamma$
  and $\Delta$ are not neat by using a simple base change trick. See
  the proof of the last part of Corollary \ref{cor.main} for an
  illustration.
\end{remark}

Section \ref{s.Ex} is devoted to constructing examples, see in
particular, Theorem \ref{t.nont} and Remark \ref{list}.

\
\begin{remark} \label{r.cpt}%
  Our definition of groups of Hermitian type forces all the simple
  factors of $G$ to be noncompact. For the purpose of constructing the
  locally symmetric varieties $M_{\Gamma}$, for $\Gamma$ an arithmetic
  subgroup of $G$, this is not necessary, and such varieties are
  commonly studied, e.g., in the theory of Shimura varieties. However,
  if $\mathbf{G}$ is simple and $G$ has a compact factor, then
  $M_{\Gamma}$ is compact~\cite[Lemma 3.2 on page 469]{BailyBorel} and
  in this case our methods do not apply.
\end{remark}

As we mentioned earlier, Farb, Kisin and Wolfson have been able to
prove incompressibility of congruence covers in many instances where
$D$ has no nontrivial rational boundary components; i.e, where
$M_{\Gamma}$ is compact.  So, we feel that, taken together,
Theorem~\ref{thm.main2} and~\cite{FKW19} provide enough evidence to
motivate the following conjecture.  It is a specialization of
Conjecture~\ref{c.vhs} from the introduction, but we feel that it is
worthwhile here to make it separately in the context of locally symmetric
varieties.

\begin{conjecture} \label{conj.main} Let $p$ be a prime, $D$  a
  Hermitian symmetric domain, $G = \Aut^0(D)$ and $\Gamma$  an
  arithmetic subgroup of $G$ as above. Then there exists a normal
  subgroup $\Delta\unlhd \Gamma$ of finite index such that the Galois
  cover $M_{\Delta} \to M_{\Gamma}$ is $p$-incompressible.  That is,
viewing $M_{\Delta}$ as a $\Gamma/\Delta$-variety, we have
  $\ed(M_{\Delta}; p) = \dim(D)$. 
\end{conjecture}

\begin{remark}
  Using Corollary~\ref{cor.versal} and the method of
  Corollary~\ref{cor.non-versal}, it is not hard to see that the
  covers $M_{\Delta}\to M_{\Gamma}$ studied in Theorem~\ref{thm.main2}
  are rarely $p$-versal.
\end{remark}

\subsection{The example of 
\texorpdfstring{$\mathcal{A}_g$}{Ag}}  
\label{ss.ag}
One typical example where we can apply Theorem~\ref{thm.main2} occurs
in the case of the moduli space of principally polarized abelian
varieties considered above when
$\mathbf{G}=\mathbf{PSp}_{2g,\mathbb{Q}}$ and $D$ is the Siegel
upper half-space (which is a tube domain, see Example
\ref{ex.siegel}). In fact, this is the example which motivated this
paper. It also concretely illustrates several of the issues that we
will have to deal with later in a more abstract setting, so we explain
it here.

For each positive integer $N\geq 3$, set
\begin{equation}
  \label{cong-sbgp}
  \Gamma_N:=\ker(\mathbf{Sp}_{2g}(\mathbb{Z})\arr \mathbf{Sp}_{2g}(\mathbb{Z}/N)).
\end{equation}
(Here we view $\mathbf{Sp}_{2g}$ as the group of symplectomorphisms of
the lattice $\mathbb{Z}^{2g}$ with a fixed symplectic form---a
reductive group scheme over $\Spec\mathbb{Z}$.)  Since $N\geq 3$, the
restriction of the central isogeny
$\mathbf{Sp}_{2g,\mathbb{Q}}\to \mathbf{G}$ to $\Gamma_N$ is
injective.  So $\Gamma_N$ is naturally isomorphic to its image
$\overline{\Gamma}_N$ in $\mathbf{G}(\mathbb{Q})$.  Moreover, this
image is contained in $G=\mathbf{G}(\mathbb{R})_+$.  (This follows
from the fact that $\Sp_{2g}(\mathbb{R})$ is connected.)

The variety $M_{\overline{\Gamma}_N}$ is the coarse moduli space
$\mathcal{A}_{g,N}$ of principally polarized $g$-dimensional abelian
varieties with level $N$-structure, and, if $p$ is a prime not
dividing $N$, the map $\pi:\mathcal{A}_{g,pN}\to\mathcal{A}_{g,N}$ is
an $\Sp_{2g}(\mathbb{F}_p)$-torsor.  In other words, if we set
$\Gamma:=\overline{\Gamma}_N$ and $\Delta:=\overline{\Gamma}_{pN}$,
then the map $\mathcal{A}_{g,pN}\to \mathcal{A}_{g,N}$ can be
identified with the map $M_{\Delta}\to M_{\Gamma}$.

Let $\widetilde{\mathbf{P}}$ denote a parabolic in
$\Sp_{2g,\mathbb{Q}}$ corresponding to a maximal isotropic subspace
for the symplectic form on $\mathbb{Q}^{2g}$.  For definiteness, 
 fix
the symplectic form $\phi$ (on $\mathbb{Z}^{2g}$) given by the matrix
$$
\begin{pmatrix}
0 & -I\\
I & 0
\end{pmatrix},
$$
where $I$ denote the $g\times g$ identity matrix.  It is not hard to see that the 
unipotent radical of $\widetilde{\mathbf{P}}$ consists of matrices of the form
\begin{equation}\label{e.uf}
\begin{pmatrix}
    I & T\\
    0 & I
\end{pmatrix}
\end{equation}
where $T$ is a symmetric $g\times g$ matrix.  So the unipotent radical 
of $\widetilde{\mathbf{P}}$ is abelian, and is, therefore, its own center. 

Let $\mathbf{P}$ denote the image of $\widetilde{\mathbf{P}}$ in
$\mathbf{G}$, and let $P:=G\cap\mathbf{P}$. Then $P$ is a maximal
parabolic and corresponds to a unique boundary component $F$ (which is
actually a single point in $\ovl D$). In fact, in the tube domain
description of $D$ (see Example \ref{ex.siegel}), this $P$ is exactly
the $P$ in Theorem \ref{thm.main2}. In particular, 
$U(F)$ is equal to the abelian algebraic group defined by the
matrices in \eqref{e.uf}.  The group $\mathbb{U}=U(F)\cap
\Gamma/U(F)\cap\Delta$ from Theorem~\ref{t.main-pure} is just the
group $\Sym^2 \mathbb{F}_p^g$, the vector space of symmetric $2\times
2$-matrices over
$\mathbb{F}_p$ viewed as a group.  So we get the following result
recovering \cite[Theorem 2]{FKW19}:

\begin{corollary} \label{cor.main} Any smooth $\mathbb{U}$-equivariant
  compactification $\overline{\mathcal{A}}_{g,pN}$ of
  $\mathcal{A}_{g, pN}$ has a $\mathbb{U}$-fixed point. Consequently,
  the $\Sp_{2g}(\mathbb{F}_p)$-torsor
  $\mathcal{A}_{g, pN}\to \mathcal{A}_{g, N}$ described above is
  $p$-incompressible. Furthermore, for any $n>1$ and $p \mid n$, the
  map $\mc{A}_{g,n} \to \mc{A}_g$ is $p$-incompressible.
\end{corollary}

\begin{proof}[Proof of Corollary~\ref{cor.main} assuming Theorem~\ref{thm.main2}]
  The Siegel upper half-space is a tube domain so, since
  $\mathbb{U} \subset \Sp_{2g}(\mathbb{F}_p)$ and
  $\mathrm{rank} (\mathbb{U}) = \dim \mathcal{A}_{g, pN}$, the first
  and second parts of the corollary follow from
  Theorem~\ref{thm.main2}.

  To prove the second part, we may assume that $p=n$. Let $N > 2$ be
  any integer such that $p \nmid N$. The natural finite surjective
  morphism from $\mc{A}_{g,pN}$ to the normalisation $\mc{A}_{g,pN}'$
  of the component of the fibre product
  $\mc{A}_{g,p} \times_{\mc{A}_g} \mc{A}_{g,N}$ dominating
  $\mc{A}_{g,N}$ is $\mb{U}$-equivariant, so it follows from the first
  part that any $\mb{U}$-equivariant compactification of
  $\mc{A}_{g,pN}'$ has a $\mb{U}$-fixed point.  By Proposition
  \ref{prop.fixed-point}, the map $\mc{A}_{g,pN}' \to \mc{A}_{g,N}$ is
  $p$-incompressible, from which we deduce that the same holds for the
  map $\mc{A}_{g,n} \to \mc{A}_g$.
\end{proof}

\subsection{Arithmetic and Congruence Subgroups}
\label{s.AC}
Here we recall some standard facts and terminology about congruence subgroups.
Our main references are~\cite{SerreCS, Rag04}.

Let $\mathbf{G}$ be a linear algebraic group defined over $\mathbb{Q}$ and
let $\mathbb{A}^f$ denote the ring of finite rational adeles.  The
embedding of $\mathbb{Q}$ in $\mathbb{A}^f$ induces an embedding of
$\mathbf{G}(\mathbb{Q})$ in $\mathbf{G}(\mathbb{A}^f)$, and we use
this embedding to regard the first group (of rational points) as a
subgroup of the second group (of adelic points).

A subgroup $\Gamma\le \mathbf{G}(\mathbb{Q})$ is said to be a
\emph{congruence subgroup} if $\Gamma=K\cap \mathbf{G}(\mathbb{Q})$
for a compact open subgroup $K$ of $\mathbf{G}(\af)$.  A subgroup
$\Gamma\le \mathbf{G}(\mathbb{Q})$ is said to be \emph{arithmetic} if
it is commensurable with a congruence subgroup.  Here two subgroups
$A$ and $B$ of an abstract group $G$ are \emph{commensurable} if
$[A:A\cap B]$ and $[B:A\cap B]$ are
 both finite.

Let $T_c$ denote the subspace topology induced on $\mathbf{G}(\mathbb{Q})$ by
its inclusion in $\mathbf{G}(\mathbb{A}^f)$.  So  $T_c$ has a neighborhood
basis of the identity consisting of congruence subgroups.  Let $T_a$
denote the topology on $\mathbf{G}(\mathbb{Q})$ obtained by taking the
arithmetic subgroups as a neighborhood basis of the identity. Then the
topology $T_a$ is a priori at least as fine as the topology $T_c$.  The
congruence subgroup problem asks whether they are the same.  In other
words, the congruence subgroup problem for $\mathbf{G}$ asks whether every arithmetic
subgroup of $\mathbf{G}(\mathbb{Q})$ is congruence.  For $\mathbf{G}$ solvable this is
known to be the case~\cite[p.~108]{Rag76}.  On the other hand, it is
not the case for $\mathbf{G}=\mathbf{SL}_2$~\cite{SerreCS}, a fact which was apparently
already known to F.~Klein.  Moreover, for $G$ semisimple, it is
never the case unless $\mathbf{G}$ is simply connected~\cite[1.2c]{SerreCS}.
See~\cite{Rag04} for a more complete account of what is known.

If $\Gamma\le G(\mathbb{Q})$, we say that a subgroup
$\Delta$ of $\Gamma$ is \emph{defined by congruence conditions} if
$\Delta=\Gamma\cap K$ for some compact open subgroup of
$G(\mathbb{A}^f)$.  Equivalently, $\Delta$ is open in $\Gamma$
for the subspace topology induced by $T_c$.

The group $\mathbf{GL}_n(\mathbb{A}^f)$ is a locally compact Hausdorff
space, and, if $\rho:\mathbf{G}\to\mathbf{GL}_{n,\mathbb{Q}}$ is a
faithful linear representation, then the induced map
$\rho(\mathbb{A}^f):\mathbf{G}(\mathbb{A}^f)\to\mathbf{GL}_n(\mathbb{A}^f)$
is a homeomorphism of $\mathbf{G}(\mathbb{A}^f)$ onto its 
image, which is necessarily closed.  It follows that the topology $T_c$
on $\mathbf{G}(\mathbb{Q})$ has a basis of open neighborhoods of $1$ of the form
$\rho^{-1}(\Phi_{n,d})$, where $\Phi_{n,d}$ denotes the set of
$n\times n$ matrices in $\mathbf{GL}_n(\mathbb{Z})$ which are congruent
to the identity modulo $d$.

\subsection{Tube domains}% 
\label{s.proof-of-thm2}

All the facts about tube domains that we use in this section can be
found in \cite[Chapter X]{Faraut-Koranyi}.
\begin{definition} \label{d.tube}
A \emph{tube domain} is a set of the form
\begin{equation}
  \label{e.tube}
  D = U + iC \subset U_{\mathbb{C}} 
\end{equation}
where $U$ is a finite dimensional real vector space and $C$ is an open
homogeneous self-adjoint cone in $U$. Then $D$ is an open subset of
the complexification $U_{\mathbb{C}}$ of $U$ so it has a natural
structure of complex manifold; in fact, it is always a Hermitian
symmetric domain \cite[Theorem X.1.1]{Faraut-Koranyi}.
\end{definition}
A tube domain is said to be \emph{irreducible} if the cone $C$ cannot
be written as a product of two cones in a nontrivial way. Any tube
domain can be written as a product of irreducible tube domains in an
essentially unique way \cite[Proposition III.4.5]{Faraut-Koranyi}.

The group $U$ acts holomorphically on $D$ by translations, so it is a
subgroup of $\Aut(D)^0$.  In fact, it is the unipotent radical of a
parabolic subgroup $P$ of $\Aut(D)^0$. The tube domain is irreducible
iff $P$ is a maximal parabolic.

We say that a pair $(\mathbf{G}, \mathbf{P})$ with $\mathbf{G}$ a
semisimple algebraic group over $\mathbb{Q}$ and $\mathbf{P}$ a
parabolic subgroup is of \emph{tube type} if there is an isomorphism
of $G := \mathbf{G}^{\ad}(\mathbb{R})_+$ with $\Aut(D)^0$, for $D$ a
tube domain as above, such that $\mathbf{P}(\mathbb{R}) \cap G$
corresponds to $P$.

\begin{example} \label{ex.siegel}
  The basic example of a tube domain is the Siegel upper half-space
  $\mathfrak{H}_g$. In this case, $U$ is the space of
  symmetric real $g \times g$ matrices, $C$ is the cone of positive
  definite matrices and the group $\mathbf{PSp}_{2g}(\mathbb{R})$ acts on
  $\mathfrak{H}_g$ by
  \[
    \gamma \cdot \Omega = (A \Omega + B)\cdot (C \Omega + D)^{-1}
  \]
  for
   $\gamma =
  \bigl( \begin{smallmatrix}
     A & B \\ C &  D
   \end{smallmatrix} \bigr) \in \mathbf{Sp}_{2g}(\mathbb{R})$ (with
   $A,B,C,D$ real $g \times g$ matrices) and
   $\Omega \in \mathfrak{H}_g$. The parabolic subgroup corresponding
   to this presentation is the subgroup $P$ consisting of elements
   $\gamma$ as above with $C = 0$. The group $U$, acting
   by translation on $\mathfrak{H}_g$, is naturally identified with
   the unipotent radical of $P$.
\end{example}

Although not needed for the proof of Theorem \ref{thm.main2}, we give
a characterisation of Hermitian symmetric domains which are of tube
type.

\begin{lemma}\label{l.td}
  Let $D=G/K$ be a Hermitian symmetric domain with boundary component
  $F$ and let $N(F)$, $W(F)$ and $U(F)$ be as in \S\ref{lss}.  Then
  the following are equivalent:
  \begin{enumerate}
  \item $\dim U(F)=\dim D$;
  \item $D$ is biholomorphic to a tube domain (as in Definition
    \ref{d.tube}) in such a way that $N(F)$ corresponds to $P$
    (equivalently, $W(F)$ corresponds to $U$, the unipotent radical of
    $P$).
  \end{enumerate}
  Furthermore, if (1) or equivalently (2) holds, then $F$ is a
  $0$-dimensional boundary component.
\end{lemma}

{\bf Note:} In Lemma~\ref{l.td}, $\dim D$ denotes the dimension of $D$ as a complex manifold,
  and $\dim U(F)$ denotes the dimension of $U(F)$ as a real Lie group.

  \begin{proof}
 The reduction to the case that $D$ is irreducible is easy. So we assume
 $D$ is irreducible and leave the reduction to the reader.

 For any boundary component $F$, by \cite[III, (4.1)]{amrt} there is a
 real analytic isomorphism $D \cong F \times C(F) \times W(F)$, where
 $C$ is a self-adjoint homogenous open cone in $U(F)$. It follows that
 $\dim U(F)=\dim D$ iff $F$ is $0$-dimensional and $W(F) =
 U(F)$. Assuming this is the case, it follows from \cite[III, Lemma
 4.7]{amrt} that $D$ is biholomorphic to the tube domain corresponding to
 the cone $C(F)$ (and $U(F) = W(F)$ corresponds to $U$).

 Conversely, if $D$ is a tube domain as defined above and if we take
 $F$ to be the boundary component corresponding to the maximal
 parabolic subgroup $P$, then $U(F) = U$ so
 $\dim U(F) = \dim D$.
\end{proof}
\begin{remark} For the explicit classification of tube domains, the
  reader may consult \cite[X.5 ]{Faraut-Koranyi}. The irreducible ones
  correspond to simple Lie groups $G$ of Hermitian type with (real)
  root system of type $C_r$, where $r = \mathrm{rank}(G)$.
\end{remark}

% For the rest of this subsection we assume that
% $(\mathbf{G}, \mathbf{P})$ is a pair of tube type with $\mathbf{G}$
% adjoint.

\begin{lemma}\label{l.cs}
  Let $(\mathbf{G}, \mathbf{P})$ be a pair of tube type with
  $\mathbf{G}$ adjoint.  Suppose that
  $\Gamma\leq G \cap\mathbf{G}(\mathbb{Q})$ is an arithmetic subgroup.
  Let $p$ be a prime number. Then there is a normal subgroup $\Delta$
  of $\Gamma$ defined by congruence conditions such that the p-torsion
  subgroup of $H:=(\Gamma\cap U)/(\Delta\cap U)$ has rank $\dim U$.
\end{lemma}

\begin{proof}
  Set $U_{\Gamma}:=\Gamma\cap U$.  The group $pU_{\Gamma}$ is of
  finite index in $U_{\Gamma}$, so it is itself an arithmetic
  subgroup of $U$.  In fact, since $U$ is unipotent, any arithmetic
  subgroup of $U$ is congruence.  Therefore, $pU_{\Gamma}$ is a
  congruence subgroup of $U$.

  Now, by~\cite[Proposition 4.2]{PR}, we can find a positive integer
  $n$ and a faithful linear linear representation
  $\rho:\mathbf{G}\to\mathbf{GL}_{n,\mathbb{Q}}$ such that $\Gamma$
  is a finite index subgroup of $\rho^{-1}\mathbf{GL}_n(\mathbb{Z})$.
  Since $pU_{\Gamma}$ is a congruence subgroup of $U$, there exists a
  positive integer $d$ such that
  $\rho^{-1}(\Phi_{n,d})\cap U\le pU_{\Gamma}$.  So set
  $\Delta=\Gamma\cap \rho^{-1}\Phi_{n,d}$. Since
  $\Phi_{n,d}\unlhd\mathbf{GL}_n(\mathbb{Z})$, $\Delta\unlhd \Gamma$.
  So the lemma follows.
\end{proof}

\begin{lemma} \label{l.cone} %
  Let $U$ be a real vector space and $C \subset U$ an open convex
  cone. Let $L \subset U$ be a lattice.  Let
  $p: U_{\mathbb{C}} \to T := U_{\mathbb{C}}/L$ be the quotient map
  and $\ovl{T} \subset T$ the partial compactification of $T$ defined
  by a maximal dimensional smooth rational polyhedral cone
  $C' \subset C $. For any $c \in U$, let $D_c \subset U_{\mathbb{C}}$
  be the open set
  \[
    \{ u \in U_{\mathbb{C}} \mid \im(u) - c \in C \} \ .
  \]
  For any sufficiently small polydisc $\ovl{S}$ centred at the torus
  fixed point $t_0$ of $\ovl{T}$, $S = \ovl{S} \cap p(D_c)$ is a
  product of punctured polydiscs.
\end{lemma}

\begin{proof}
  We may choose coordinates such that
  $L = \mathbb{Z}^n \subset \mathbb{R}^n = U$ for some $n > 0$,
  $U_{\mathbb{C}} = \mathbb{C}^n$, and $T =
  (\mathbb{C}^{\times})^n$. We may also assume that
  \[
   C' =  \{ (x_1,x_2,\dots, x_n) \in \mathbb{R}^n \ |\  x_i \geq 0 \ \forall\  i\}
 \]
 and then $\ovl{T} = \mathbb{C}^n$, with
 $p: \mathbb{C}^n \to (\mathbb{C}^{\times})^n$ given by
  \[
    (z_1,z_2,\dots,z_n) \mapsto (e^{2\pi iz_1},e^{2\pi iz_2},\dots,
    e^{2\pi i z_n}) \ .
  \]
  Since $C' \subset C$, and $C$ is open and convex, $\{  (z_1,z_2,\dots,z_n) \in \mathbb{C}^n \ \mid \im(z_i) \gg 0 \ \forall \
  i\} \subset D_c$. It follows that $p(D_c)$ contains
  \[
    \{ (w_1,w_2,\dots,w_n) \in (\mathbb{C}^{\times})^n \ \mid \ |w_i| < \epsilon \
    \forall \ i\}
  \]
  if $\epsilon$ is sufficiently small.
\end{proof}

\begin{proof}[Proof of Theorem~\ref{thm.main2}]

  Let $X = M_{\Delta}$, $Y = M_{\Gamma}$, $K:=\Gamma/\Delta$, and
  $f:X \to Y$ the natural map making $X$ into a $K$-torsor over $Y$.
  Let $\ovl{Y}$ be the Baily--Borel \cite{BailyBorel} compactification
  of $Y$ and $\ovl{X}$ a smooth $\mathbb{U}$-equivariant
  compactification of $X$ such that there is a morphism
  $\bar{f}:\ovl{X} \to \ovl{Y}$ extending $f$.

  The group $U \cap \Gamma$, which is a lattice in $U$, gives rise to
  an algebraic torus $T = (U \cap \Gamma)\backslash U_{\mathbb{C}}$
  containing $(U \cap \Gamma)\backslash D$. Let $C' \subset C$ be a
  smooth rational polyhedral cone of maximal dimension, where the
  integral structure defining smoothness is given by $U \cap
  \Gamma$. The cone $C'$ gives rise to a torus embedding
  $T \subset \ovl{T}$ with $\ovl{T}$ an affine space. Since
  $C' \subset C$, the intersection of a polydisc $\ovl{S}$ around
  $t_0$, the torus fixed point of $\ovl{T}$, with $\ovl{S}$ is a
  product $S$ of punctured polydiscs (apply Lemma \ref{l.cone} with
  $c=0$) and the map $\pi_1(S) \to \pi_1(T) = (U \cap \Delta)$ is an
  isomorphism.

  The natural map $(U\cap \Gamma)\backslash D \to \Gamma\backslash D$
  induces a holomorphic map $g: S \to Y$ with the property that the
  induced map on $\pi_1$ corresponds to the inclusion
  $U \cap \Gamma \to \Gamma$. Therefore, the image of the map
  $g_*: \pi_1(S) \to K = \Gamma/\Delta$ is the finite abelian group
  $\mathbb{U}$.  By the Borel extension theorem \cite{BorelExt}, the
  map $g$ extends to a holomorphic map $\bar{g}: \ovl{S} \to \ovl{Y}$.
  By Proposition \ref{p.criterion}, $\mathbb{U}$ has a fixed point in
  $\ovl{X}$, so this proves part (1) of the theorem.

  Part (2) of the theorem follows from this by
  Proposition~\ref{prop.fixed-point} and part (3) of the theorem is
  the content of Lemma \ref{l.cs}.
\end{proof}

\subsection{General Hermitian symmetric domains} \label{s.gherm}

% \begin{remint}
% Some more precise references need to be added below.
% \end{remint}

General Hermitian symmetric domains do not have as simple a
description as do tube domains, but the theory of Siegel domains of
the third kind (see, e.g.,~\cite[III \S4]{amrt}) allows us to give a
proof of Theorem \ref{t.main-pure} which is very similar to the proof of
Theorem \ref{thm.main2}.

Let $D$, $G$, $\mathbf{G}$ be as in \S\ref{lss} and let $F$ be a rational
boundary component. We let $N(F)$ be the normalizer of $F$, $W(F)$ its
unipotent radical and $U(F)$ the centre of $W(F)$.

We have already mentioned the Harish-Chandra embedding, which is a
holomorphic embedding of $D$ in $\mathfrak{p}_+$.  Recall that $D=G/K$
where $K$ is a maximal compact subgroup of $G$ (the stabilizer of a
given element of $D$).  Let $G_{\mathbb{C}}$ (resp. $K_{\mathbb{C}}$)
denote the complexification of $G$ (resp. $K$)~\cite[III \S 6,
Proposition 20]{BourbakiLie1-3}. The subspaces $\mathfrak{p}_{\pm}$ of
$\mathfrak{p}_{\mathbb{C}}$ are abelian subalgebras of $\mathfrak{g}$
corresponding to subgroups $P_{\pm} =\exp \mathfrak{p}_{\pm}$ of
$G_{\mathbb{C}}$.  Moreover, $K_{\mathbb{C}}$ normalizes $P_{-}$ and
the resulting subgroup $K_{\mathbb{C}}P_{-}$ is parabolic.  Set
$\Dcech:=G_{\mathbb{C}}/K_{\mathbb{C}}P_{-}$; this space, a complex
generalized flag variety, is called the \emph{compact dual} of $D$.
The exponential map $\exp:\mathfrak{p}_+\to G_{\mathbb{C}}$ followed
by the quotient map to $\Dcech$ then gives rise to an open immersion
of $\mathfrak{p}_+$ in $\Dcech$~\cite[Theorem III.2.1]{amrt}.  The
composition of the two embeddings
$D\hookrightarrow \mathfrak{p}_+\hookrightarrow \Dcech$ is known as
the \emph{Borel embedding}.

\subsubsection{}
All the results stated in this subsection can be found in \cite[III
\S4.3]{amrt} (esp.~pp 152-153) and the references therein.

Let $D(F) = U(F)_{\mathbb{C}} \cdot D \subset\Dcech$. This is a
submanifold of $\Dcech$ on which there is an action of
$U(F)_{\mathbb{C}}$ such that $D$ is preserved by $U(F)$. The group
$U(F)_{\mathbb{C}}$ acts freely (and holomorphically) on $D(F)$ and
the quotient by this action is a complex manifold $D(F)'$, so we have
a map $\pi_F': D(F) \to D(F)'$ whose fibres are principal homogenous
spaces over $U(F)_{\mathbb{C}}$. Furthermore, $D(F)'$ is contractible.

For any $w \in D(F)'$, we may identify $(\pi_F')^{-1}(w)$ with $U(F)$
by choosing any point as the origin. Then
\begin{equation} \label{e.Dw}
 D_w := (\pi_F')^{-1}(w)\cap D= \{u \in U(F)_{\mathbb{C}} | \im(u) \in C(F) - c(w)\}
\end{equation}
where $C(F)$ is an open self-adjoint homogenous cone in $U(F)$ and
$c(w) \in U(F)$.

\subsubsection{}

\begin{proof}[Proof of Theorem~\ref{t.main-pure}]
  The proof is very similar to the case of tube domains.

  Set $U := U(F)$.  The inclusion of $D$ in $D(F)$ induces an
  inclusion
  $(\Gamma \cap U)\backslash D \to (\Gamma \cap U)\backslash D(F)$
  which induces an isomorphism on fundamental groups. Since the map
  $\pi_F'$ is equivariant for the action of $U(F)_{\mathbb{C}}$, it is
  also equivariant for the action of $U \cap \Gamma$. For any
  $w \in D(F)'$, the inclusion
  $(U \cap \Gamma)\backslash (\pi_F')^{-1}(w) \to (U \cap
  \Gamma\backslash D(F))$ also induces an isomorphism on fundamental
  groups. We now apply Lemma \ref{l.cone} to get
  $S \subset (U \cap \Gamma)\backslash D_w$, a product of punctured
  discs, and a sequence of maps
\[
  S \to (U \cap \Gamma) \backslash D_w \to (U \cap \Gamma)\backslash D
  \to \Gamma \backslash D.
\]
The discussion above shows that the image of the map on fundamental
groups induced by the composite of these maps is $U \cap \Gamma$. We
may now complete the proof as in the case of tube domains by using the
Baily--Borel compactification of $D/\Gamma$ and the Borel extension
theorem.
\end{proof}

\subsection{Incompressibility in positive characteristic}
\label{pos}

In many cases the locally symmetric varieties $M_{\Gamma}$ have a
modular interpretation which leads to a natural model defined over a
well-defined number field $L$. In fact, they often have natural smooth
models $\mathcal{M}_{\Gamma}$ over a localisation $R$ of the ring of
integers of $L$, so we can reduce $\mathcal{M}_{\Gamma}$ modulo
maximal ideals $P$ of $R$. When $M_{\Delta}$ also has such a model, we
get a finite \'etale covering
$\mathcal{M}_{\Delta, k} \to \mathcal{M}_{\Gamma,k}$, where $k =
R/P$. It is then natural to ask when these covers are incompressible
or $p$-incompressible (over $\bar{k}$).

Our proof of incompressibility in the case of the moduli space of
curves was characteristic free, but for locally symmetric varieties
the proof was complex analytic and does not immediately extend to
fields of positive characteristic: for example, in this case there is
no analogue of Borel's extension theorem, even for
$\mathcal{A}_g$. However, the theory of toroidal compactifications of
integral models allows us to bypass this difficulty by using the
existence of fixed points in characteristic zero to get fixed points
(on suitable compactifications) over $k$ as well, from which we can
deduce incompressibility using the fixed-point method. All this is
best done using the language of Shimura varieties, however, rather
than explaining this in detail we give a proof of the analogue of
Corollary \ref{cor.main} and then point out the references which can
be used to generalize this result.

\smallskip

Let $k$ be an algebraically closed field of characteristic $l>0$. The
varieties $\mathcal{A}_{g,N}$ are moduli spaces of $g$-dimensional
principally abelian varieties with level $N$ structure, so can be
defined over $k$ as long as $l \nmid N$ \cite{git}. If $p$ is a prime
not dividing $N$ and $l \neq p$, then the cover
$\mathcal{A}_{g,pN}/k \to \mathcal{A}_{g,N}/k$ is defined and is an
$\Sp(2g,\mathbb{F}_p)$-torsor if $N \geq 3$. Here we use $/k$ to
emphasize that the varieties are over the field $k$. The theorem below
extends \cite[Theorem 2]{FKW19} to fields of positive characteristic.

\begin{theorem} \label{c.agp}%
  The $\Sp_{2g}(\mathbb{F}_p)$-torsor
  $\mathcal{A}_{g,pN}/k \to \mathcal{A}_{g,N}/k$ is $p$-incompressible
  if $p\nmid N$, $N \geq 3$ and $l \nmid pN$.  Furthermore, for any
  $n>1$ and prime $p$ such that $p \mid n$ and $(n, p\car{k}) = 1$,
  the map $\mc{A}_{g,n} \to \mc{A}_g$ is $p$-incompressible.

\end{theorem}

\begin{proof}
  Let $R$ be the ring $\mathbb{Z}[1/n, \zeta_{n}]$, where $\zeta_{n}$
  is a primitive $n$-th root of unity in $\mathbb{C}$, and assume
  $n \geq 3$. By results of Mumford \cite{git}, there exists a smooth
  scheme $\mathcal{A}_{g,n}/R$ whose fibre at any prime $P$ of $R$ is
  the variety $\mathcal{A}_{g,n}/k_P$, where $k_P$ denotes the residue
  field at $P$.  By \cite[IV, Theorem 6.7]{FaltingsChai} there exist
  smooth proper algebraic spaces $\overline{\mathcal{A}}_{g,n}/R$
  (depending on some auxiliary data which we suppress) which contain
  $\mathcal{A}_{g,n}/R$ as a fibrewise dense open subspace and such
  that the natural action of
  $\mathbf{Sp}_{2g}(\mathbb{Z}/n\mathbb{Z})$ on $\mathcal{A}_{g,n}/R$
  extends to $\overline{\mathcal{A}}_{g,n}/R$.  Furthermore, by
  \cite[V, Theorem 5.8]{FaltingsChai}, the auxiliary data may be
  chosen so that $\overline{\mathcal{A}}_{g,n}/R $ is projective, so
  in particular, it is a scheme.

  We now take $n = pN$. Since
  $\overline{\mathcal{A}}_{g,pN}/\mathbb{C}$ is a smooth
  compactification of $\mathcal{A}_{g,pN}/\mathbb{C}$, it follows from
  the discussion in \S\ref{ss.ag}, that the group
  $\mathbb{U} \subset \mathbf{Sp}_{2g}(\mathbb{F}_p) \subset
  \mathbf{Sp}_{2g}(\mathbb{Z}/pN\mathbb{Z})$ has a fixed point in
  $\mathcal{A}_{g,pN}(\mathbb{C})$. Since
  $\overline{\mathcal{A}}_{g,pN}/R$ is proper, by specialisation it
  follows that $\mathbb{U}$ has a fixed point in
  $\mathcal{A}_{g,pN}(k)$.  Since $\mathbb{U}$ is contained in
  $ \mathbf{Sp}_{2g}(\mathbb{F}_p)$, the Galois group of the cover
  under consideration, and
  $\mathrm{rank}(\mathbb{U}) = \dim \mathcal{A}_{g, pN}/k$, the
  theorem follows from the fixed point method (Proposition
  \ref{prop.fixed-point}).

  To prove the second part of the theorem, we may assume that $p=n$.
  If $p > 2$, then $\ovl{\mc{A}}_{g,p}/k$ is smooth and the above
  shows that it has a $\mb{U}$-fixed point so the statement follows
  from the fixed point method. Now suppose $p=2$, let $N> 2$ be any
  integer such that $(N, p\car{k}) = 1$, and consider the morphism
  $\mc{A}_{g,pN}/k \to \mc{A}_{g,pN}'/k$ as in the proof of Corollary
  \ref{cor.main}. Since $p=2$ this is an isomorphism, so the statement
  follows from the second part of the theorem.
 \end{proof}

 \begin{remark} \label{r.hodgetype} %
   As mentioned earlier, versions of Theorems \ref{t.main-pure} and
   \ref{thm.main2} can be proved for certain more general locally
   symmetric varieties, by essentially the same argument as above,
   using integral toroidal compactifications of Shimura varieties of
   Hodge type; see \cite{Lan-PEL} for the PEL case and
   \cite{Madapusi-tor} for the general Hodge type setting. Here we
   need to assume that $\Gamma$ is a congruence subgroup which is
   hyperspecial at the prime $l$ in order to get a model which is
   smooth: this corresponds to the condition $l \nmid N$ above.  We
   leave this as an exercise for the interested reader. However, we
   note that the theory of integral toroidal compactifications is not
   needed to prove these results for the reductions modulo primes of
   large (undetermined) residue characteristic: this follows
   immediately from the results in characteristic zero by ``spreading
   out'' any smooth proper equivariant compactification and applying
   the fixed point method to the reduction modulo such a prime.
\end{remark}

\section{Incompressible hyperspecial congruence covers}\label{s.Ex}  
Pick a prime $p$, which we will keep fixed for this section.  Our goal
is to prove generalizations of some of the results
of~\cite[\S4]{FKW19} producing congruence covers with group
$\mathbb{G}(\mathbb{F}_q)$ where the $\mathbb{G}$ are certain
semisimple algebraic groups over $\mathbb{F}_q$ with $q=p^r$ for some
positive integer $r$.  The main theorem here is Theorem~\ref{t.nont}.
As explained in Remark~\ref{list}, this allows us to produce
congruence covers for most, but not all, of the classical groups
$\mathbb{G}$ considered in~\cite{FKW19}.  The main new result is the
existence of $p$-incompressible congruence covers for locally
symmetric varieties of type $E_7$ with congruence group also of type
$E_7$.

The main technical tools and references are as follows:
\begin{enumerate}
\item Results from SGA3~\cite{SGA3-3} used to control the reduction
  modulo $p$ of the subgroup scheme $\mathbf{U}(F)\leq \mathbf{G}$
  associated to the group $U(F)$ of Theorem~\ref{t.main-pure}.
\item A well-known approximation result for number fields, Proposition
  \ref{GWtors} below.
\item A theorem of Prasad and Rapinchuk on producing isotropic groups
with specific behavior at a set of primes~\cite{pr06}.
\end{enumerate}

\subsection{General notation}

Suppose $\mathbf{H}$ is an algebraic group over $\mathbb{Q}$, and let
$\af:=\mathbb{A}_{\mathbb{Q}}$ denote the finite rational adeles.  If
$p$ is a prime number, we write $\afp$ for the prime to $p$ adeles. So
$\af=\afp\times\mathbb{Q}_p$.

If $K$ is a compact open subgroup of $\mathbf{H}(\af)$ we set
$\Gamma_K:=K\cap\mathbf{H}(\mathbb{Q})$ and
$\Gamma_K^+=K\cap\mathbf{H}(\mathbb{Q})_+$.  We say that $K$ is neat
if it is neat in the sense of Pink~\cite{pink}.  For $K$ neat,
$\Gamma_K$ is neat as well.

We say that $\mathbf{H}$ has \emph{strong approximation} if $\mathbf{H}(\mathbb{Q})$ is dense in $\mathbf{H}(\af)$.   This obviously implies that 
$\mathbf{H}(\mathbb{Q}) K=\mathbf{H}(\af)$ for any compact open subgroup
$K\leq \mathbf{H}(\af)$.    

\begin{proposition}\label{p.ca}
  Suppose that $\mathbf{H}$ is either 
\begin{enumerate}
\item[(a)] a simply connected semisimple
  algebraic group over $\mathbb{Q}$ without compact $\mathbb{Q}$-simple
factors, or
\item[(b)]  a Cartesian power
  $\mathbb{G}_a^n$ of the additive group.
\end{enumerate}
Then 
\begin{enumerate}
\item the Lie group $\mathbf{H}(\mathbb{R})$ is connected;
\item $\mathbf{H}$ has strong approximation.  
\end{enumerate}
\end{proposition}
\begin{proof}
  Assertion (1) is obvious in case (b) and assertion (2) follows from
  the case $n=1$, which is the usual strong approximation for the
  adeles.  In case (a), strong approximation is proven
  in~\cite[Theorem 7.12]{PR} and assertion (1) is due to E.~Cartan.
  See~\cite[Corollaire 4.7]{BorelTitsComp} or~\cite[Proposition
  7.6]{PR}.
\end{proof}

\subsection{Smooth \texorpdfstring{$\mathbb{Z}_p$}{Zp}-models and principal 
\texorpdfstring{$p$}{p}-pairs}
\label{s.prinp} Suppose $\mathbf{H}$ is an algebraic group of
$\mathbb{Q}$.  A smooth $\mathbb{Z}_p$-model of $\mathbf{H}$ is a
smooth scheme $\mathbf{H}_x$ over $\mathbb{Z}_p$ together with an
isomorphism
$\mathbf{H}_{\mathbb{Q}_p}\cong
\mathbf{H}_x\times_{\mathbb{Z}_p}\mathbb{Q}_p$.  If $\mathbf{H}$ is
reductive (resp. semisimple), then a reductive (resp. semisimple)
$\mathbb{Z}_p$-model is a smooth model which is also a reductive
(resp. semisimple) group scheme over $\mathbb{Z}_p$.

Given a smooth $\mathbb{Z}_p$-model $\mathbf{H}_x$, suppose $K^p\leq\mathbf{H}(\afp)$
is a compact open subgroup.   Define compact open subgroups $K_p$ and $L_p$
of $\mathbf{H}(\mathbb{Q}_p)$ by setting
\begin{align*}
K_p&=\mathbf{H}_x(\mathbb{Z}_p),  \\
L_p&=\ker(\mathbf{H}_x(\mathbb{Z}_p)\to \mathbf{H}_x(\mathbb{F}_p)).
\end{align*}
Then set $K=K^p\times K_p$ and $L=K^p\times L_p$.  Clearly, $L\unlhd K$. 
Moreover, since $\mathbf{H}_x$ is smooth, reduction mod $p$ induces an
isomorphism
\begin{equation}\label{e.smooth}
K/L\cong \mathbf{H}_x(\mathbb{F}_p).
\end{equation}
For future reference, we call $(K,L)$ the \emph{principal $p$-pair} arising
from the smooth $\mathbb{Z}_p$-model $\mathbf{H}_x$ and the compact open
subgroup $K^p\leq\mathbf{H}(\afp)$.  

\begin{remark}\label{r.neat}
Given $\mathbf{H}_x$, it is easy to see that there exists an open neighborhood
$V$ of the identity in $\mathbf{H}(\afp)$ such that $K$ is neat as long 
as $K^p\subseteq V$.   In particular, as long as $\mathbf{H}$ has a smooth $\mathbb{Z}_p$-model, there exists principal $p$-pairs $(K,L)$ with $K$ neat. 
\end{remark}

\begin{proposition}\label{p.co2dis}
  Suppose $\mathbf{H}$ is an algebraic group over $\mathbb{Q}$ with strong 
approximation and with $\mathbf{H}(\mathbb{R})$ connected.  Let 
$(K,L)$ be a principal $p$-pair arising from $\mathbf{H}_x$ and $K^p$ as above. 
Then inclusion induces an isomorphism of groups
$\Gamma_K/\Gamma_L\cong K/L$. 
\end{proposition}
\begin{proof}
The homomorphism $\Gamma_K/\Gamma_L\to K/L$ is obviously one-one by the
definition of $\Gamma_L$.   Since $\mathbf{H}$ has strong approximation,
$\mathbf{H}(\mathbb{Q})L=\mathbf{H}(\af)$.  So, for any 
$k\in K$, we can find 
$\ell\in L$ and $h\in\mathbf{H}(\mathbb{Q})$ such that 
$k=h\ell$.  Then $h=\ell k^{-1}\in H(\mathbb{Q})\cap K=\Gamma_K$.
It follows that $\Gamma_K/\Gamma_L\to K/L$ is onto.  
\end{proof}

If $\mathbf{H}$ is a semisimple group, then $\mathbf{H}$ has a
semisimple $\mathbb{Z}_p$-model if and only if the
$\mathbb{Q}_p$-group $\mathbf{H}_{\mathbb{Q}_p}$ is quasisplit and
split over an unramified extension of $\mathbb{Q}_p$~\cite{TitsCorv}.
If these two conditions hold, then $\mathbf{H}_{\mathbb{Q}_p}$ is
called an \emph{unramified group} and we say that $\mathbf{H}$ is
\emph{unramified at $p$}.  Isomorphism classes of semisimple
$\mathbb{Z}_p$-models of $\mathbf{H}$ are then in one-one
correspondence hyperspecial points $x$ in the Bruhat--Tits building of
$\mathbf{H}_{\mathbb{Q}_p}$.  This motivates our notation for smooth
models.

\subsection{Congruence subgroups for simply connected
  groups}\label{s.csub}

For the rest of this section we fix a simply connected group
$\mathbf{G}$ over $\mathbb{Q}$ of Hermitian type (as in \S\ref{lss}).
We write $\rho:\mathbf{G}\to\mathbf{G}^{\ad}$ for the canonical homomorphism
to the adjoint group, which we call the \emph{adjoint homomorphism}. 

For $K\leq \mathbf{G}(\mathbb{A}^{f,p})$ a compact open subgroup, we
set $\Gamma_K^{\ad}=\rho(\Gamma_K)$.  By Proposition~\ref{p.ca},
$\mathbf{G}(\mathbb{R})$ is connected.  So $\Gamma_K^{\ad}$ is an
arithmetic subgroup of $\mathbf{G}^{\ad}(\mathbb{Q})_+$.  Analogously
to \S\ref{lss}, we set $M_K =\Gamma_K^{\ad}\backslash D$.
\begin{lemma}\label{l.nqts}
  Suppose $L$ and $K$ are compact open subgroups of $\mathbf{G}(\afp)$ with 
$L\unlhd K$ and with $K$ neat.  Then 
\begin{enumerate}
\item The adjoint homomorphism $\rho$ induces an isomorphism
$\Gamma_K\to \Gamma_K^{\ad}$.
\item We have 
$$
\Gamma_{K}^{\ad}/\Gamma_L^{\ad}\cong \Gamma_K/\Gamma_L\cong K/L.
$$
\item The natural morphism $M_L \to M_K$ is a finite \'etale Galois
  cover with Galois group $K/L$.
\end{enumerate}
\end{lemma}
\begin{proof}
  (1) Since $K$ is neat, $K$ does not meet the center of $\mathbf{G}$.
It follows that $\rho_{|\Gamma_K}:\Gamma_K\to \Gamma_K^{\ad}$ is one-one.
But this homomorphism is onto by the definition of $\Gamma_K^{\ad}$. 
\smallskip

(2) This 
follows from (1) and Proposition~\ref{p.co2dis}. 
\smallskip

(3) Since $K$ is neat, $\Gamma_K$ is torsion free.  So, by (1),
$\Gamma_K^{\ad}$ is torsion free as well.  It follows that 
$M_L \to M_K$ is \'etale and Galois with Galois group  
$\Gamma_K^{\ad}/\Gamma_L^{\ad}=K/L$. 
\end{proof}

\begin{corollary}\label{c.nqts}
  Suppose $\mathbf{G}$ is unramified at a prime $p$, 
$x$ is a hyperspecial point of the Bruhat--Tits building of
$\mathbf{G}_{\mathbb{Q}_p}$, $K^p$ is a compact open subgroup of 
$\mathbf{G}(\af)$ and $(K,L)$ is the principal $p$-pair arising 
from this data.  Assume that $K$ is neat.  Then $M_L \arr M_K$
is an \'etale  Galois cover with Galois group $\mathbf{G}_x(\mathbb{F}_p)$.  
\end{corollary}
\begin{proof}
  By Lemma~\ref{l.nqts}, $M_L \arr M_K$ is \'etale and Galois
  with Galois group $K/L$.  But, under the hypotheses,
  $K/L\cong\mathbf{G}_x(\mathbb{F}_p)$.
\end{proof}

\subsection{Boundary components and reduction modulo 
\texorpdfstring{$p$}{p}}
Now fix a rational boundary component $F$ of $D$ and let $N(F)$,
$W(F)$, $U(F)$ be as before.  Write $\mathbf{N}^{\ad}(F)$ (resp.
$\mathbf{W}^{\ad}(F), \mathbf{U}^{\ad}(F)$) for the corresponding
algebraic subgroups of $\mathbf{G}^{\ad}$, and write $\mathbf{N}(F)$
for the inverse image of $\mathbf{N}^{\ad}(F)$ in $\mathbf{G}$, a
parabolic subgroup.  Write $\mathbf{W}(F)$ for the unipotent radical
of $\mathbf{N}(F)$ and $\mathbf{U}(F)$ for the center of
$\mathbf{W}(F)$.  Then $\rho$ induces a isomorphisms
$\mathbf{W}(F)\cong \mathbf{W}^{\ad}(F)$
(resp. $\mathbf{U}(F)\cong\mathbf{U}^{\ad}(F)$.  Moreover, the group
of real points of $\mathbf{W}(F)$ (resp. $\mathbf{U}(F)$) is
isomorphic to the Lie group $W(F)$ (resp. $U(F)$).
 
Since $F$ will be fixed in this section, we allow
ourselves to drop it from the notation writing, for example, $\mathbf{N}$
instead of $\mathbf{N}(F)$.  

Suppose further that $\mathbf{G}$ is unramified and fix a hyperspecial
point $x$ giving us a group scheme $\mathbf{G}_x$ as above in
\S\ref{s.prinp}.  As in~\cite[Expos\'e XXVI]{SGA3-3}, write
$\Par\mathbf{G}_x$ for the $\mathbb{Z}_p$-scheme representing the
functor of parabolic subgroup schemes of $\mathbf{G}_x$.  Then
$\Par\mathbf{G}_x$ is proper over $\mathbb{Z}_p$.  So it follows that
the parabolic subgroup $\mathbf{N}$ extends uniquely to a parabolic
subgroup scheme $\mathbf{N}_x$ of $\mathbf{G}_x$.

Write $\mathbf{W}_x$ for the unipotent radical of 
$\mathbf{N}_x$~\cite[XXII.5.11.4]{SGA3-3}.  This is a 
closed subgroup scheme of $\mathbf{N}_x$ whose geometric
fibers are connected and unipotent with the property that 
$\mathbf{N}_x/\mathbf{W}_x$ is reductive.  In particular,
it is an extension to $\mathbb{Z}_p$ of $\mathbf{W}$.

\begin{proposition}\label{p.fromsga}
  The subgroup $\mathbf{U}$ of $\mathbf{W}$ extends to a smooth
  central closed subgroup scheme $\mathbf{U}_x$ of $\mathbf{W}_x$.
  Moreover, if we set $\mathbf{V}_x=\mathbf{W}_x/\mathbf{U}_x$ we have
\begin{align*}
\mathbf{U}_x&\cong \mathbb{G}_a^{r_U};\\
\mathbf{V}_x&\cong \mathbb{G}_a^{r_V}
\end{align*}
with  $r_U=\dim\mathbf{U}$ and $r_V=\dim \mathbf{W}/\mathbf{U}$.
\end{proposition}
\begin{proof}
First note that, we have an exact sequence of unipotent algebraic groups
over $\mathbb{Q}$, 
\begin{equation}
\label{e.UandV}
1\to\mathbf{U}\to \mathbf{W}\to \mathbf{V}\to 1
\end{equation}
where both $\mathbf{U}$ and $\mathbf{V}$ are abelian. 
Moreover, the Lie algebras of $\mathbf{U}$ and $\mathbf{W}$ are defined
in terms of root spaces relative to a suitable maximal 
$\mathbb{R}$-split torus of $\mathbf{G}$. (See~\cite[p.~143]{amrt}.)

Now, from~\cite[Expos\'e 26, Proposition 2.1]{SGA3-3}, it follows 
that $\mathbf{W}_x$ admits a finite filtration by closed subgroup schemes
\begin{equation}\label{sga3-26.2.1}
\mathbf{W}_0=\mathbf{W}_x \supseteq \mathbf{W}_1\supseteq\mathbf{W}_2 
\supseteq \cdots \supseteq \mathbf{W}_n=\{1\}
\end{equation}
where the quotients $\mathbf{W}_i/\mathbf{W}_{i+1}$ are 
group schemes associated to vector bundles on $\mathbb{Z}_p$.  
So, since any vector bundle on $\mathbb{Z}_p$ is trivial,
for each $i$, we have $\mathbf{W}_i/\mathbf{W}_{i+1}\cong \mathbb{G}_a^{r_i}$ for
some  nonnegative
integer $r_i$.

In fact, when $\mathbf{G}$ is pinned the vector bundles are root spaces
of $\mathbf{G}$ and, in the general case, the result is deduced by descent.
Examining the proof, one sees that, given $\mathbf{W}_x$ as above,
there are at most two nontrivial vector bundles involved and  that $\mathbf{U}_x$ is 
a central subgroup scheme of $\mathbf{W}_x$ with 
$\mathbf{U}_x\otimes_{\mathbb{Z}_p}\mathbb{Q}_p=\mathbf{U}$.  The proposition
then follows easily.  
\end{proof}

\begin{theorem}\label{t.Hgp}
  Suppose $\mathbf{G}$ is unramified at $p$ and $(K,L)$ is the principal
$p$ pair associated to a smooth $\mathbb{Z}_p$-model $\mathbf{G}_x$ and a
compact open subgroup $K^p\leq\mathbf{G}(\afp)$.  Set 
$$
\mathbb{U}:=\frac{\Gamma_K\cap\mathbf{U}(\mathbb{Q})}{\Gamma_L\cap\mathbf{U}(\mathbb{Q})}
\leq \frac{\Gamma_K}{\Gamma_L}.
$$
Then $\mathbb{U}$ is an $\mathbb{F}_p$-vector space with 
$$
\dim_{\mathbb{F}_p} \mathbb{U} = r_U.
$$
\end{theorem}
\begin{proof}
  Since $\mathbf{U}_x$ is a closed subgroup scheme of $\mathbf{G}_x$
  with generic fiber $\mathbf{U}$, Proposition~\ref{p.fromsga} implies
that 
  $\mathbf{U}(\mathbb{Q}_p)\cap K_p=\mathbf{U}(\mathbb{Q}_p)\cap\mathbf{G}_x(\mathbb{Z}_p)=\mathbf{U}_x(\mathbb{Z}_p)\cong \mathbb{Z}_p^{r_U}$.

Let $R_p:=\ker[\mathbf{U}_x(\mathbb{Z}_p)\to \mathbf{U}_x(\mathbb{F}_p)]$.  
Then, we have a commutative diagram of short exact sequences
$$
\begin{tikzcd}
  1\ar[r] & R_p \ar[r]\ar[d, hook] & \mathbf{U}_x(\mathbb{Z}_p)\ar[r]\ar[d, hook] 
          & \mathbf{U}_x(\mathbb{F}_p)\ar[r]\ar[d, hook] & 1\\
  1\ar[r] & L_p\ar[r]    & \mathbf{G}_x(\mathbb{Z}_p)\ar[r] 
          & \mathbf{G}_x(\mathbb{F}_p)\ar[r] & 1                         
\end{tikzcd}
$$
where the vertical arrows are monomorphisms and the nontrivial arrows
on the right are reduction module $p$. From this, it follows that
$R_p=L_p\cap\mathbf{U}_x(\mathbb{Z}_p)$.  But then, since
$L_p\leq K_p$ and
$K_p\cap\mathbf{U}(\mathbb{Q}_p)=\mathbf{U}_x(\mathbb{Z}_p)$, we have
$R_p=L_p\cap\mathbf{U}(\mathbb{Q}_p)$ as well.

 Set $K^p_U:=K^p\cap\mathbf{U}(\afp)$, 
$K_U=K\cap\mathbf{U}(\af)$ and $L_U=L\cap\mathbf{U}(\af)$.
Then $K^p_U$ (resp. $K_U, L_U$) is a compact
open subgroup of $\mathbf{U}(\afp)$ (resp. $\mathbf{U}(\af)$).  

Moreover, 
\begin{align*}
K_U&=(K^p\times K_p)\cap\mathbf{U}(\af)=
(K^p\times K_p)\cap (\mathbf{U}(\afp)\times\mathbf{U}(\mathbb{Q}_p))
\\
&=(K^p\cap\mathbf{U}(\afp))\times (K_p\cap\mathbf{U}(\mathbb{Q}_p))
=(K^p\cap\mathbf{U}(\afp)\times\mathbf{U}_x(\mathbb{Z}_p).  
\end{align*}

Similarly,
\begin{align*}
L_U&=(K^p\times L_p)\cap\mathbf{U}(\af)=
(K^p\cap\mathbf{U}(\afp))\times (L_p\cap\mathbf{U}(\mathbb{Q}_p))\\
&=
(K^p\cap\mathbf{U}(\afp))\times R_p.
\end{align*} 

It follows that $K_U/L_U\cong \mathbf{U}_x(\mathbb{Z}_p)/R_p\cong \mathbb{F}_p^{r_U}$.  

We have $\Gamma_{K_U}=\mathbf{U}(\mathbb{Q})\cap K\cap 
\mathbf{U}(\af)=\mathbf{G}(\mathbb{Q})\cap K\cap \mathbf{U}(\af)=
\Gamma_K\cap\mathbf{U}(\mathbb{Q})$.  And, similarly, 
$\Gamma_{L_U}=\Gamma_L\cap\mathbf{U}(\mathbb{Q})$.  

Now apply Proposition~\ref{p.co2dis} to deduce that 
$\mathbb{U}=\Gamma_{K_U}/\Gamma_{L_U}\cong K_U/L_U$. 
\end{proof}

\subsection{Hermitian pairs adapted to semisimple groups over
  \texorpdfstring{$\mathbb{F}_p$}{Fp}}

\label{s.cong}

Suppose $p$ is a prime number and $\mathbb{G}$ is a simply connected
semisimple group over $\mathbb{F}_p$ for some prime $p$.  Our goal is to find a simply connected group $\mathbf{G}$ of
Hermitian type with a $0$-cuspidal parabolic subgroup $\mathbf{P}$
such that
\begin{enumerate}
\item $\mathbf{G}$ is unramified at $p$.
\item For some, hence any, reductive model $\mathbf{G}_x$ of $\mathbf{G}$
over $\mathbb{Z}_p$, we have $\mathbf{G}_x\otimes_{\mathbb{Z}_p}\mathbb{F}_p\cong\mathbb{G}$.
\end{enumerate}

We will say that such a pair $(\mathbf{G}, \mathbf{P})$ is 
\emph{adapted} to $\mathbb{G}$.

\begin{proposition}
  Suppose $(\mathbf{G},\mathbf{P})$ is adapted to $\mathbb{G}$. Then
  there exists a neat principal $p$-pair $(K,L)$ in
  $\mathbf{G}(\mathbb{A}^f)$ with $\Gamma_K/\Gamma_L\cong \mathbb{G}$.
  Consequently, the cover 
\begin{equation}\label{ppcov}
  M_L\to M_K
\end{equation} is finite \'etale with Galois
group $\mathbb{G}(\mathbb{F}_p)$ and with 
\begin{equation}\label{edin}
\ed_{\mathbb{G}(\mathbb{F}_p)}M_L\geq \dim U(F)
\end{equation}
where $F$ is any $0$-dimensional boundary component.  If
$(\mathbf{G}, \mathbf{P})$ is of tube type then the cover
\eqref{ppcov} is $p$-incompressible.
\end{proposition}
\begin{proof}
  Using Remark~\ref{r.neat}, we can find a neat principal $p$-pair
  $(K,L)$.  As in \S\ref{s.prinp}, this gives rise to the \'etale
  Galois cover \eqref{ppcov}.  Then, by Theorem~\ref{t.Hgp}, we have
  $\dim_{\mathbb{F}_p}\mathbb{U}=\dim U(F)$.  Then~\eqref{edin}
  follows from Theorem~\ref{t.main-pure}, and the incompressibility
  for tube domains follows from Lemma~\ref{l.td}, or directly from
  Theorem \ref{thm.main2}.
\end{proof}

\subsection{Reduction and Statement of Theorem} The following lemma
reduces the problem of finding an adapted pair
$(\mathbf{G}, \mathbf{P})$ to the case where $\mathbb{G}$ is almost
simple over $\mathbb{F}_p$.

\begin{lemma}\label{lem-prod}
  Suppose $(\mathbf{G}_i, \mathbf{P}_i)$ are adapted to $\mathbb{G}_i$
  for $i=1,2$.  Then
  $(\mathbf{G}_1\times\mathbf{G}_2,\mathbf{P}_1\times \mathbf{P}_2)$
  is adapted to $\mathbb{G}_1\times\mathbb{G}_2$.
\end{lemma}
\begin{proof}
  This follows by noting that if $\mathbf{P}_1$ and $\mathbf{P}_2$ are
  $0$-cuspidal, then so is $\mathbf{P}_1 \times \mathbf{P}_2$.
\end{proof}

We can write any simply connected, semisimple, algebraic group $H$
over a field $L$ as
\begin{equation}
  \label{Gprod}
  H=\prod_{i=1}^k\Res^{K_i}_{L} H_i
\end{equation}
where the $K_i$ are finite separable extensions of $L$, the $H_i$ are
absolutely almost simple, simply connected groups and $\Res$ denotes
Weil restriction.  Moreover, the description of $H$ in~\eqref{Gprod}
is essentially unique in that the list of $K_i$ and $H_i$ appearing is
unique up to reordering and isomorphisms.  Let us say that a
semisimple group $H$ over a finite field is of \emph{potentially
  Hermitian type} if the following two conditions are satisfied:
\begin{enumerate}
\item None of the $H_i$ in \eqref{Gprod} are of type $E_8$, $F_4$ or $G_2$.
\item None of the $H_i$ are triality forms of $D_4$. 
\end{enumerate}

\begin{theorem}\label{t.nont}
  Suppose that $\mathbb{G}$ is a simply connected group of potentially
  Hermitian type over $\mathbb{F}_p$.  Then there exists a
  $0$-cuspidal Hermitian pair $(\mathbf{G},\mathbf{P})$ adapted to
  $\mathbb{G}$.
\end{theorem}

\begin{remark} \label{list}
  $ $
  \begin{enumerate}
  \item It will follow from the proof of the theorem and the
    classification of tube domains \cite[p.~213]{Faraut-Koranyi} that
    we can choose the pair $(\mathbf{G},\mathbf{P})$ to be of tube
    type in the cases that $\mathbb{G}$ has no factors of triality
    type $D_4$, no factors over $\overline{\mathbb{F}}_p$ of type
    $A_r$, with $r$ even, and no factors of type $E_6$.  In
    particular, we may take $(\mathbf{G},\mathbf{P})$ to be of tube
    type when $\mathbb{G}$ is of type $E_7$.

    When $(\mathbf{G},\mathbf{P})$ is of tube type, by combining
    Theorem \ref{thm.main2}, Theorem \ref{t.Hgp} and Theorem
    \ref{t.nont} we get $p$-incompressible congruence covers with
    Galois group $\mathbb{G}(\mathbb{F}_p)$.
  \item In contrast to (1), in \cite{FKW19} split factors of type
    $A_r$ with $r$ even are also allowed, but factors of type $E_7$
    are not. However, we do get weaker results for groups of type
    $A_r$ with $r$ even.
  \item If $\mathbb{G}$ is a form of $E_6$, the dimension of the
    Hermitian symmetric domain corresponding to $\mathbf{G}$ is $16$,
    but we only get a lower bound of $8$---the dimension of the centre
    of the unipotent radical of $\mathbf{P}$---for the $p$-essential
    dimension of congruence covers with Galois group
    $\mathbb{G}(\mathbb{F}_p)$.
  \end{enumerate}
\end{remark}

\subsection{Satake--Tits Index} Associated to any reductive
group $\mathbf{G}$ over a field $L$ we have the Satake--Tits 
index~\cite[p. 38]{TitsBoulder}. This
consists of the following data:
\begin{enumerate}
\item The Dynkin diagram $D=\Dyn \mathbf{G}$.
\item An action $\tau:\Gal(L_{\sep}/L)\to \Aut D$ of the absolute Galois group of $L$
on the Dynkin diagram, sometimes called the $*$-action. 
\item A collection of $D_0$ of vertices called the \emph{uncircled vertices}. 
\end{enumerate}
The vertices of $D$, which are  defined in terms of simple roots, can
be identified with conjugacy classes of maximal parabolic subgroups in
$\mathbf{G}_{L_{\sep}}$.  Conjugation then gives the $*$-action.
The uncircled vertices correspond to simple roots in the anisotropic
kernel of $\mathbf{G}$. Then conjugacy classes of maximal parabolic
subgroups defined over $L$ are in one-one correspondence with $\tau$
orbits in $D\setminus D_0$. These are the orbits that Tits calls
\emph{distinguished}, and, in the Satake--Tits index these orbits are drawn
with circles around them. 

If $K/L$ is a field extension, then we can identify $D$ with the Dynkin
diagram $D_K$ of $\mathbf{G}_K$.  So the first ingredient in the
Satake--Tits index is insensitive to field extension.  However, the
$*$-action obviously changes: for example, if $K\subset L_{\sep}$,
then the $*$-action $\tau_K:\Gal(L_{\sep}/K)\to \Aut D$ is obtained by
restriction.  If we write $D_{K,0}$ for the set of uncircled
vertices in the Satake--Tits index of $\mathbf{G}_K$, then we have 
$D_{K,0}\subseteq D_0$.   

In the case that $L=\mathbb{R}$ and $(\mathbf{G},\mathbf{P})$ is a
$0$-cuspidal pair, the $*$-action of $\Gal(\mathbb{C}/\mathbb{R})$ on
$D=D_{\mathbb{R}}$ is the opposition on the Dynkin diagram. It
preserves connected components of $D$ and, for $\mathbf{G}$ almost
simple, it is
\begin{enumerate}
\item the unique nontrivial action in types $A$, $D_n$ for $n$ odd,
and in type $E_6$,
\item the trivial action in all other types.
\end{enumerate}
  
If $\mathbf{G}$ is almost simple, then Deligne's special vertex $s$
(see~\cite[p.~258]{DeligneShimuraCorv}) lies in $D_{\mathbb{R},0}$.
Its orbit under the $\Gal(\mathbb{C}/\mathbb{R})$ $*$-action
corresponds to conjugacy class of real parabolics associated to the
$0$-dimensional boundary components.  For general $\mathbf{G}$, there
is a special vertex in each component of $D$, whose orbit is circled
in the real Satake--Tits index.  Write $Z_{\infty}$ for the union of
the orbits of these special vertices under the real $*$-action.  Then
there exists a $0$-dimensional rational boundary component if and only
if $Z_{\infty}$ is contained in the circled vertices of the
Satake--Tits index of $\mathbf{G}$.

\subsection{Lemmas} 
We believe the following result on approximation for number fields 
is probably well-known, but for the convenience of the reader and lack of a
published reference we give a proof following ideas we learned
from~\cite{SaltmanGenAIM}.

\begin{proposition}\label{GWtors}  Suppose $n$ is a positive integer, 
$F$ is a number field and $S$ is a finite set of places of $F$. Then the 
map 
$$
\coh^1(F,S_n)\to \prod_{\nu\in S} \coh^1(F_{\nu}, S_n)
$$
is surjective. 
\end{proposition}

\begin{proof}  An $S_n$-torsor or, equivalently, a degree $n$ \'etale algebra, over a field $K$ is determined by a monic, separable
  polynomial of degree $n$ over the field.  So write $U_n(K)$ for the
  set of such polynomials.   We get a surjective map $U_n(K)\twoheadrightarrow \coh^1(K,S_n)$
taking a polynomial $p$ to the isomorphism class of the \'etale algebra $K[x]/p$. 

Now $U_n(F_{\nu})$ is an
  open subset of the space $F_{\nu}^{n-1}$ of all monic degree $n$
  polynomials.  This gives rise to a commutative diagram
\begin{equation}
\begin{tikzcd}
  U_n(F)\ar[r]\ar[d, two heads] & \prod_{\nu\in S} U_n(F_{\nu})\ar[d, two heads]\\
 \coh^1(F,S_n)\ar[r] & \prod_{\nu\in S} \coh^1(F_{\nu},S_n)
\end{tikzcd}
\end{equation}
where the top horizontal arrow is just (the diagonal) inclusion and the bottom horizontal
arrow is restriction.  

Now, it is well-known and easy to see that the fibers of the right vertical
arrow are open.  (In fact, the bottom right set is finite and the right vertical
arrow is continuous.)   And it follows from weak approximation that the image
of the top arrow is dense.  From these two facts along with the surjectivity
of the downward arrows, the proposition follows directly.
\end{proof}

\begin{remark}
  In~\cite[Theorem 5.8]{SaltmanGenAIM}, Saltman gives variants of this
  argument which are applicable to finite groups $G$ admitting a
  generic Galois extension.  In particular, this holds for any group $G$ 
admitting a faithful linear representation $V$ such that $V\sslash G$ is rational. 
\end{remark}

\begin{lemma}\label{GW}
  Suppose $p$ is a prime number and $q=p^n$ with $n\in\mathbb{Z}_+$.
  Then there exists a totally real extension $L/\mathbb{Q}$ of degree
  $n$ in which the prime $p$ is inert.
\end{lemma}
\begin{proof}
  This follows immediately from Proposition~\ref{GWtors}.
\end{proof}

The next lemma applies an approximation theorem of Prasad and
Rapinchuk~\cite{pr06} to produce isotropic algebraic groups
$\mathbf{H}$ with specific behavior at places of $L$.

\begin{lemma}\label{lem-PR}
  Suppose $\mathbb{H}$ is an absolutely almost simple, simply connected
group over $\mathbb{F}_{q}$ which is of potentially Hermitian type. 
Let $L/\mathbb{Q}$ be a field extension
as in Lemma~\ref{GW}. Write $S_{\infty}$ for the set of real places of $L$
and $\{\nu\}$ for the place lying above $p$.   
Then there exists a simply connected group 
$\mathbf{H}$ over $L$ such that 
\begin{enumerate}
\item $\mathbf{H}_{\omega}$ is of Hermitian symmetric type  
for each place $\omega\in S_{\infty}$.
\item $\mathbf{H}$ is unramified at $\nu$, and, for some, hence any,
reductive model $\mathbf{H}_x$ of $\mathbf{H}$ over the ring $R_{\nu}$ 
of integers in $L_{\nu}$, we have $\mathbf{H}_x\otimes_{R_{\nu}} \mathbb{F}_q\cong \mathbb{H}$.  
\item The special vertices of $\mathbf{H}_{\omega}$ are all the
  same in the Dynkin diagram $\Dyn\mathbf{H}$ of $\mathbf{H}$.
Moreover, 
  the orbit $Z_{\infty}$ of this vertex under the opposition is
stable under the $\Gal(L)$ action and  circled in the 
Satake--Tits index of $\mathbf{H}$ over $L$.
\end{enumerate}
\end{lemma}
\begin{proof} Let $D$ denote the Dynkin diagram of $\mathbb{H}$ with
  vertices $\alpha_1,\ldots, \alpha_r$ corresponding to the simple
  roots as in the diagrams in~\cite[p.~532]{helga}.  Notice that, if
  $\mathbb{H}$ is of type $D_4$, then, owing to our assumption that
  $\mathbb{H}$ is not triality, we can assume that $\alpha_1$ is fixed
  by action of $\Gal(\mathbb{F}_q)$ on $D$.

  Now we want to choose a simply connected real group
  $\mathbf{H}_{\infty}$ for each possible type of $\mathbb{H}$ as
  follows. (For most of the types there is no real choice, but we want
  to choose carefully in type $A$ and $D$.)
\begin{enumerate}
\item[(a)] If $\mathbb{H}$ is type $A_r$, then $\mathbf{H}_{\infty}$
  if $\mathbf{SU}_{\bigl ( \tfrac{r+1}{2},\tfrac{r+1}{2}\bigr )}$ is
$r$ is odd and $\mathbf{SU}_{\bigl (\tfrac{r}{2}, \tfrac{r}{2} +1 \bigr)}$ if $r$
is even.
\item[(b)] If $\mathbb{H}$ is of type $B_r$, then we take $\mathbf{H}_{\infty}$ 
to be the unique form of Hermitian symmetric type: $\mathbf{Spin}_{2r-1,2}$.
\item[(c)] If $\mathbb{H}$ is of type $C_r$ we take $\mathbf{H}_{\infty}$ to
be the symplectic group $\mathbf{Sp}_{2r}$. 
\item[(d)] If $\mathbb{H}$ is of type $D_r$, then we take $\mathbf{H}_{\infty}$
to be the group $\mathbf{Spin}_{2r-2,2}$. 
\item[(e)] In types $E_6$ and $E_7$, we take $\mathbf{H}_{\infty}$ to
  be the unique form of Hermitian type.
\end{enumerate}
In each of these choices, we also choose an identification of $\Dyn\mathbf{H}_{\infty}$ with $D$. We can do this so that the labels on the vertices 
in~\cite[p.~532]{helga} match up with the labels $\alpha_1,\ldots, \alpha_r$
for $D$.  This is especially important for $\mathbb{H}$ of type 
$D_4$, where it implies that Deligne's special vertex, which is in that
case $\alpha_1$,
is fixed by the action of $\Gal(\mathbb{F}_p)$ on $D$. 

Now, note that we have 
\begin{equation}
  \label{e.AutD}
  \Aut D =
  \begin{cases}
    \mathbb{Z}/2, & \text{for $\mathbb{H}$ of type $A$, $E_6$ or $D_r$ with $r\neq 4$;}\\
            S_3, & \text{for $\mathbb{H}$ of type $D_4$;}\\
            \{1\}, & \text{otherwise.}
  \end{cases}
\end{equation}

Let $V$ denote the subgroup of $\Aut D$ stabilizing the
$\Gal(\mathbb{C}/\mathbb{R})$-orbit of the special vertex in
$\mathbf{H}_{\infty}$.  It is easy to see that $V=\Aut D$ itself
except in the case where $\mathbb{H}$ is of type $D_4$.  In that case,
$V$ is the subset of $\Aut D=S_3$ stabilizing $\alpha_1$.  It follows
from our nontriality assumption that, in any case, the *-action of
$\Gal(\mathbb{F}_q)$ on $D$ factors through $V$.  Using Proposition
\ref{GWtors} for the group $V$, we can find a quasisplit group
$\mathbf{H}^{\mathrm{qs}}$ over $L$, which is split by a quadratic
extension $M/L$ with $\Gal(M/L)$ acting on $D$ through $V$, such that
$\mathbf{H}^{\mathrm{qs}}$ is an inner form of $\mathbf{H}_{\infty}$
at each infinite place of $L$, and, at the prime $\nu$,
$\mathbf{H}^{\mathrm{qs}}$ is an unramified group over $R_{\nu}$ with
reduction modulo $\nu$ equal to $\mathbb{H}$.  For example, if
$\mathbb{H}$ is of type $D_r$ for $r$ odd, then $\mathbf{H}_{\infty}$
is of outer type over $\mathbb{R}$.  If $\mathbb{H}$ is split, then
producing $\mathbf{H}^{\mathrm{qs}}$ is equivalent to finding a
totally imaginary quadratic extension $M/L$ in which $\nu$ splits.  On
the other hand, if $\mathbb{H}$ is the unique nonsplit simply
connected group over $\mathbb{F}_q$ of type $D_r$, then we need to
find a totally imaginary quadratic extension $M/L$ in which $\nu$ is
inert.  As Proposition \ref{GWtors} allows us to produce quadratic
extensions of $L$ with arbitrary behaviour at the $\omega$ and at
$\nu$, we can produce such groups $\mathbf{H}^{\mathrm{qs}}$.

Since the $\Gal(L)$ action on $\mathbf{H}^{\mathrm{qs}}$ factors through $V$, the orbit $Z_{\infty}$ of Deligne's special vertex in $D=\Dyn\mathbf{H}_{\infty}$
is fixed by $\Gal(L)$. 

Now, by~\cite[Theorem 1]{pr06}, we can find an inner form $\mathbf{H}$
of $\mathbf{H}^{\mathrm{qs}}$ such that
$\mathbf{H}\otimes L_{\omega}=\mathbf{H}_{\infty}$ at all real primes
$\omega$ and $\mathbf{H}\otimes R_{\nu}$ is unramified with reduction
equal to $\mathbb{H}$.  Moreover, by~\cite[Theorem 1 (iii)]{pr06}, we
can do this in such a way that $Z_{\infty}$ is contained in the set of
circled vertices for the Satake--Tits index of $D_{L}$.  This completes
the proof.
\end{proof} 

\subsection{Proof of Theorem~\ref{t.nont}} By Lemma~\ref{lem-prod},
  we can assume that $\mathbb{G}$ is almost simple.  So
  $\mathbb{G}=\Res^{\mathbb{F}_q}_{\mathbb{F}_p} \mathbb{H}$ for some
    $q=p^n$ and some absolutely almost simple group $\mathbb{H}$ of
    potentially Hermitian type over $\mathbb{F}_q$.  Let
$L/\mathbb{Q}$ be a field extension as in Lemma~\ref{GW}
and let $\mathbf{H}$ be a group as in Lemma~\ref{lem-PR}.
Set $\mathbf{G}=\Res^L_{\mathbb{Q}}\mathbf{H}$.   % Set $X=\mathbf{G}^{\ad}_+(\mathbb{R})/K$ for some maximal compact subgroup $K\leq \mathbf{G}^{\ad}_+(\mathbb{R})$.

Since Weil restriction commutes with base change, $\mathbf{G}$ is of
Hermitian type and the base-change of $\mathbf{G}$ to $\mathbb{Q}_p$
is unramified with reduction modulo $p$ isomorphic to $\mathbb{G}$.

Let $\Sigma$ be the orbit under the opposition of the set of special
vertices in the Dynkin diagram $\Dyn\mathbf{G}$.  This is the subset
of the Dynkin diagram corresponding to the $0$-dimensional boundary
components.  Since $Z_{\infty}\subset\Dyn\mathbf{H}$ is stable under
$\Gal(L)$ and contained in the circled vertices, it follows that
$\Sigma$ is stable under $\Gal(\mathbb{Q})$ and it also follows that
$\Sigma$ is contained in the circled vertices of
$\Dyn\mathbf{G}$. Therefore, $\mathbf{G}$ contains a $0$-cuspidal
parabolic $\mathbf{P}$.

% \begin{example} \label{ex.e7} There is a unique simply connected
%   form $\mathbf{G}$ of $E_7$ over $\mathbb{Q}$ which is split at all
%   primes $p$ and contains a parabolic $\mathbf{P}$ such that
%   $(\mathbf{G}, \mathbf{P})$ is of tube type. This group can be
%   constructed using octonions (see, e.g.,~\cite{BaExcept}) but we
%   explain here how this can be seen using the theorem of Prasad and
%   Rapinchuk:

%   The real form of $E_7$ of Hermitian type is an inner form of the
%   split form and has rank $3$ over $\mathbb{R}$, so using
%   \cite[Theorem 1 (iii)]{pr06} as in the proof of Lemma \ref{lem-PR},
%   we may find a form $\mathbf{G}$ of $E_7$ of rank $3$ over
%   $\mathbb{Q}$ which is of Hermitian type. For each prime $p$, the
%   rank of $\mathbf{G}_{\mathbb{Q}_p}$ is at least $3$ and the special
%   vertex must be circled in the Satake--Tits diagram over
%   $\mathbb{Q}_p$, since both statements hold over $\mathbb{Q}$ (the
%   second because it holds over $\mathbb{R}$ and the ranks are the same
%   over $\mathbb{Q}$ and $\mathbb{R}$). By consulting the table of
%   forms of $E_7$ in \cite[p.~59]{TitsBoulder} we see that $\mathbf{G}$
%   must split over $\mathbb{Q}_p$ for all $p$. The uniqueness of
%   $\mathbf{G}$ follows from the Hasse principle.
% \end{example}
% One can see by comparing \cite[p.~532]{helga} and
% \cite[Table~1.3]{GrossZ} that the symplectic groups are the only other
% absolutely simple, simply connected groups over $\mathbb{Q}$ which are
% unramified at all primes and of Hermitian type.

\subsection{The group of type \texorpdfstring{$E_7$}{E7}} \label{s.e7}%
There is a unique simply connected form of $E_7$ over $\mb{R}$ of
Hermitian type, and this group has $\mb{R}$-rank $3$
(\cite[p.~518]{helga}). As a special case of Theorem \ref{t.nont}, or
more directly from \cite{PR}, we see that there exists a form $\mf{G}$
of $E_7$ over $\mb{Q}$ of Hermitian type which has a $0$-cuspidal
parabolic $\mf{P}$ or, equivalently, has $\mb{Q}$-rank $3$.

\begin{lemma} \label{l.e7}%
  Any (simply connected) form $\mf{G}$ of $E_7$ over $\mb{Q}$ of
  Hermitian type with $\mb{Q}$-rank $3$ is split at all finite primes
  $p$. Furthermore, such a group is unique up to isomorphism and
  extends to a reductive group scheme over $\mb{Z}$ which is also
  unique up to isomorphism.
\end{lemma}
Such a group over $\mb{Z}$ was constructed explicitly by Baily in
\cite{BaExcept} using integral octonions and by Gross in \cite{GrossZ}
using Galois cohomology.

\begin{proof}
  Only the first part of the lemma is perhaps not explicitly in the
  literature. To prove this, we note that for each prime $p$, the
  $\mb{Q}_p$-rank of $\mathbf{G}$ is at least $3$ and the special
  vertex must be circled in the Satake--Tits diagram of
  $\mf{G}_{\mathbb{Q}_p}$, since both statements hold over
  $\mathbb{Q}$ (the second because it holds over $\mathbb{R}$ and the
  ranks are the same over $\mathbb{Q}$ and $\mathbb{R}$). By
  consulting the table of forms of $E_7$ in \cite[p.~59]{TitsBoulder},
  we see that the only such form of $E_7$ over any $p$-adic field is
  split.

  The uniqueness of $\mf{G}$ over $\mb{Q}$ up to isomorphism follows
  from the Hasse principal for simply connected groups (see
  \cite[p.~266]{GrossZ}) and the existence of an integral model
  follows, for example, from \cite[Proposition 1.2]{GrossZ}. The
  uniqueness of integral models follows from \cite[Proposition 2.1,
  1)]{GrossZ} and the fact that hyperspecial subgroups of
  $\mf{G}(\mb{Q}_p)$ are conjugate under
  $\mf{G}^{\mathrm{ad}}(\mb{Q}_p)$ for all primes $p$: these facts,
  together with the central exact sequence
  \[
    1 \to \mu_2 \to \mf{G}\to  \mf{G}^{\mathrm{ad}} \to 1
  \]
  of group schemes over $\mb{Q}$ imply that given any two extensions
  of $\mf{G}$ to reductive group schemes over $\mb{Z}$, there exists
  an element $g$ of $\mf{G}^{\mathrm{ad}}(\mb{Q})$ such that
  conjugation by $g$ on $\mf{G}$ extends to an isomorphism of the two
  integral models.
\end{proof}

For the rest of this section we shall use $\mf{G}$ to also denote a
reductive integral model of the $\mb{Q}$-group as in Lemma \ref{l.e7}.
The group $\Gamma:=\mf{G}(\mb{Z})$ is then an arithmetic subgroup of
$\mf{G}(\mb{R})$ (and is in fact a maximal arithmetic subgroup
\cite[Theorem 5.2]{BaExcept}). For any integer $n>1$, let $\Gamma(n)$
be the kernel of the (surjective) reduction map
$\Gamma \to \mf{G}(\mb{Z}/n\mb{Z})$. We set
$M_{\Gamma} = \Gamma \backslash D$ and
$M_{\Gamma(n)} = \Gamma(n) \backslash D$, where $D$ is the Hermitian
symmetric domain corresponding to $\mf{G}$. Note that the actions on
$D$ are not always faithful, but the quotients have a natural
algebraic structure such that the maps $M_{\Gamma(n)} \to M_{\Gamma}$
are algebraic \cite{BailyBorel}.

\begin{corollary} \label{t.e7}%
  For any integer $n>1$ and any prime $p \mid n$, the cover
  $M_{\Gamma(n)} \to M_{\Gamma}$ is $p$-incompressible.
\end{corollary}

\begin{proof}
  We may clearly assume that $n=p$.  Since the group $\Gamma$ is not
  neat we cannot immediately apply Theorem \ref{thm.main2}. However,
  by embedding $\mf{G}$ in $\mf{GL}_{r,\mb{Z}}$ for some $r$,
  we see that if $N>2$ is any integer then $\Gamma(N)$ is neat. Using
  this, and the fact that $D$ is a tube domain, the corollary follows
  in exactly the same way as Corollary \ref{cor.main}.
\end{proof}

\appendix

\section{Essential dimension of variations of Hodge structure}
\label{vhs}
Let $B$ be an irreducible variety over $\mathbb{C}$ and let $L$ be a
locally constant sheaf on $B^{\mathrm{an}}$. The purpose of this appendix is
to define the essential dimension $\ed(L)$ of $L$ and give a lower
bound for this when $L$ is the local system $\mathbb{H}_{\mathbb{Z}}$
underlying a variation of Hodge structure (or VHS, for short) $\mathbb{H}$.

Analogously to the definition of $\ed(f)$ for a generically \'etale
morphism in the introduction, we define the essential dimension
$\ed(L)$ of $L$ to be the minimum of the dimensions of irreducible
varieties $B'$ such that the following condition holds:
\begin{itemize}
\item There is a locally constant sheaf $L'$ on ${B'}^{\mathrm{an}}$,
  a dense open subvariety $U$ of $B$, and a morphism $f:U \to B'$,
  such that $L|_{U^{\mathrm{an}}} \cong (f^{\mathrm{an}})^{-1}(L')$.
\end{itemize}

  We then have the following result, which provides some weak evidence
  for Conjecture \ref{c.vhs} beyond the case of variations whose
  associated period domain is Hermitian symmetric.
  \begin{proposition} \label{p.vhs}%
    Let $B$ be a smooth irreducible variety over $\mathbb{C}$ and let
    $\mathbb{H}$ be an integral (polarised) variation of Hodge
    structure on $B$ with underlying local system
    $\mathbb{H}_{\mathbb{Z}}$.  Then
  \[
    \mathrm{ed}(\mathbb{H}_{\mathbb{Z}}) \geq d,
  \]
  where $d$ is the
  dimension of the image of the period map associated to the VHS
  $\mathbb{H}$ (as defined in \S\ref{ss.conj}).
\end{proposition}
The simple proof below was explained to us by Madhav Nori.

\begin{proof}
  Suppose $\mathrm{ed}(\mathbb{H}_{\mathbb{Z}}) = n$. It follows from
  the definitions that there exists a nonempty open subvariety
  $U \hookrightarrow B$, a variety $B'$ of dimension $n$, a morphism
  $f:U \to B'$, and a local system $L'$ on $B'$ such that
  $\mathbb{H}_{\mathbb{Z}}|_{U^{\mathrm{an}}} \cong
  (f^{\mathrm{an}})^{-1}(L')$. Replacing $B$ by $U$, we may assume
  that $f$ is defined on all of $B$. Clearly
  $(f^{\mathrm{an}})^{-1}(L)$ is constant on (the analytification of)
  all fibres of $f$, so the local system $\mathbb{H}_{\mathbb{Z}}$ is
  also constant on these fibres.  

  Let $\phi:B^{an} \to \Gamma\backslash D $ be the period map as in
  \S\ref{ss.conj}.  Then, since $\mathbb{H}_{\mathbb{Z}}$ is constant
  on the fibers of $f$, it follows from the theorem of the fixed part
  that $\phi$ is constant on the connected components of the fibers of
  $f$~\cite[Corollaire 4.1.2]{DeligneHodge2}, \cite[Theorem
  7.22]{Schmid-VHS}.  This implies that $\dim V \geq \dim Y$, where
  $Y$ is the image of the period map, i.e., $n \geq d$.
\end{proof}

\begin{remark}$ $
  \begin{enumerate}
  \item For the purposes of Proposition \ref{p.vhs}, the algebraicity
    of the image of the period map proved in \cite{BBT} is
    irrelevant. Without knowing anything about the image, we may
    define the ``dimension of the image of the period map'' simply to
    be the rank of $\mathrm{d}\phi$ at a general point of
    $B^{\mathrm{an}}$.
  \item The inequality in Proposition \ref{p.vhs} can be strict:
    $\mathbb{H}_{\mathbb{Z}}$ can be a nonconstant local system even
    if $\mathbb{H}_x$ is a constant Hodge structure. However, one can
    see that there always exists a finite \'etale cover
    $p:\tilde{B} \to B$ such that $\ed(p^*(\mathbb{H}))$ is equal to
    the dimension of the image of the period map.
\end{enumerate}
\end{remark}

\section{Essential dimension of
  \texorpdfstring{$\Sp_{2g}(\mb{F}_p)$}{Sp(2g,Fp)}}\label{a.benson}

The purpose of this appendix it to give the computation, due to Dave
Benson, of the essential dimension at $p$ of the group 
$\Sp_{2g}(\mb{F}_p)$.   

\begin{theorem}[Benson]\label{t.Benson}  Let $p$ be an odd prime and $g$ a positive
integer. Then 
$$\ed_{\mathbb{C}} (\Sp_{2g}(\mb{F}_p);p)=p^{g-1}.$$ 
\end{theorem}

\begin{proof}
Let $V:=\mathbb{F}_p^{2g}$  with basis $e_1,\ldots, e_{2g}$ and let 
$\phi$ be the symplectic form on $V$ given by 
\begin{equation}
\phi=\sum_{i=1}^g e_{2g+1-i}^*\wedge e_i^*.
\end{equation}
Then the group $G$ of linear transformations of $V$ preserving $\phi$ is
an explicit presentation of $\Sp_{2g}(\mb{F}_p)$. 

Now, $G$ has a faithful irreducible representation $M$ of degree $(p^g-1)/2$
 ~\cite[Proposition 2.7]{Ward}.  In fact, although we do not need this 
fact, $(p^g-1)/2$ is the minimum dimension
of a nontrivial irreducible representation of $G$~\cite{LandazuriSeitz}.
So $M$ is a faithful linear representation of $G$ of minimum possible dimension.

Write $P$ for the subgroup of $G$ consisting of all upper-triangular
matrices in $G$ with $1$s on the diagonal.  Then $P$ is a $p$-Sylow subgroup
of $G$.  So, by Karpenko--Merkurjev~\cite{km2}, 
$\ed (G;p)=\ed P=\dim W$, where $W$ is a faithful representation
of $P$ of minimal dimension.  

Let $H$ denote the set of matrices in $G$ with $1$s on the diagonal
but with $0$s everywhere else except the top row and rightmost column.
Then it is not hard to see that $H$ is an extraspecial $p$-group: its
center is the subgroup $Z$ consisting of matrices with $1$s on the
diagonal and $0$s everywhere else except the top-right entry.  And
$H/Z\cong C_p^{2g-2}$.  Moreover, it is easy to see that the center of
$P$ is also $Z$.

Since $P$ is a $p$-group with cyclic center, it follows that any
faithful representation of $P$ has a faithful irreducible constituent
(and similarly for $H$).  So, letting $W$ denote (as above) a faithful
representation of $P$ of minimal dimension, it follows that $W$ is
irreducible.  Since the degree of any complex character divides the
order of the group~\cite[Corollary 6.5.2, p.~52]{serre}, it follows
that $\dim W$ is a power of $p$. From this, and the fact that $G$ has
a faithful representation of degree $(p^g-1)/2$, it follows that
$\dim W=p^k$ with $k\leq g-1$.

On the other hand, since $H$ is an extraspecial $p$-group of order
$p^{2g-1}$, any faithful representation of $H$ has dimension at least
$p^{g-1}$~\cite[Theorem 5.5, p.~208]{GorensteinFG}.  It follows that $\dim W=p^{g-1}$, and this completes the
proof.
\end{proof}

\begin{remark}
  As mentioned in the introduction, in work in progress, Jesse
  Wolfson's PhD student Hannah Knight has computed
  $\ed_{\mathbb{C}}(G;p)$ for groups such as
  $G=\mathbf{Sp}_{2g}(\mathbb{F}_{p^r})$ for $r>1$ as well as for
  analogous orthogonal groups. See also~\cite{BMKS} for results which
  can be used to compute the essential dimension at $p$ of groups such
  as $\mathbf{GL}_n(\mathbb{F}_p)$.
\end{remark}

\bibliographystyle{amsalpha}
%% input <fp.bbl>
%% \bibliography{fp}
%%%%% BEGIN fp.bbl%%%%%%%%%%%%%%%
\def\cprime{$'$}\def\cprime{$'$}
\providecommand{\bysame}{\leavevmode\hbox to3em{\hrulefill}\thinspace}
\providecommand{\MR}{\relax\ifhmode\unskip\space\fi MR }
% \MRhref is called by the amsart/book/proc definition of \MR.
\providecommand{\MRhref}[2]{%
  \href{http://www.ams.org/mathscinet-getitem?mr=#1}{#2}
}
\providecommand{\href}[2]{#2}

%%%%% END  fp.bbl%%%%%%%%%%%%%%%
\end{document}